\documentclass[letterpaper, 10 pt, conference]{ieeeconf}  

\IEEEoverridecommandlockouts                              
\usepackage{amsmath}
\usepackage{amssymb}  
\usepackage{amsthm_custom}  
\usepackage[short, c2, nocomma]{optidef_custom}
\usepackage{bm}
\usepackage{comment}
\usepackage[hidelinks]{hyperref}
\usepackage{makecell}
\usepackage{subdepth}
\usepackage{mathtools}
\usepackage{caption}
\usepackage[labelformat=simple]{subcaption}

\usepackage[english]{babel}
\usepackage{blindtext}

\usepackage{pgfplots}
\usetikzlibrary{arrows.meta}
\usetikzlibrary{calc}
\usepgfplotslibrary{groupplots}
\usepgfplotslibrary{dateplot}
\usetikzlibrary{backgrounds, plotmarks}

\usepackage{mycustompgf}

\newcommand{\ddt}{\frac{\mathrm{d}}{\mathrm{d}t}}

\newcommand{\ntr}{n_{\mathrm{tr}}}
\newcommand{\dntr}{\dot{n}_{\mathrm{tr}}}
\newcommand{\lamn}{\lambda_{\mathrm{n}}}

\newcommand{\nha}{n_{1,\mathrm{h}}}
\newcommand{\nhb}{n_{2,\mathrm{h}}}

\newtheorem{proposition}{Proposition}
\newtheorem{example}{Example}
\newtheorem{lemma}{Lemma}
\newtheorem{remark}{Remark}

\makeatletter
\newlength{\negph@wd}
\DeclareRobustCommand{\negphantom}[1]{%
	\ifmmode
	\mathpalette\negph@math{#1}%
	\else
	\negph@do{#1}%
	\fi
}
\newcommand{\negph@math}[2]{\negph@do{$\m@th#1#2$}}
\newcommand{\negph@do}[1]{%
	\settowidth{\negph@wd}{#1}%
	\hspace*{\negph@wd}
}
\makeatother

\usepackage[nocompress]{cite}

\title{\LARGE \bf
High Accuracy Numerical Optimal Control for Rigid Bodies with Patch Contacts through Equivalent Contact Points -- Extended Version
}

\date{}
\author{Christian Dietz${}^{1,2}$, Armin Nurkanovi\'c${}^{2}$, Sebastian Albrecht${}^{1}$, Moritz Diehl${}^{2,3}$
	\thanks{\noindent \hspace*{-1.06em}${}^{1}$Autonomous Systems and Control, Siemens Technology, Germany \newline
		${}^{2}$Department of Microsystems Engineering (IMTEK), University of Freiburg, Germany \newline
		${}^{3}$Department of Mathematics, University of Freiburg, Germany \newline 
		Correspondent: {\tt\small dietz.christian@siemens.com} \newline
		This research was supported by BMWK via 20D2123B, 03EI4057A and 03EN3054B, by DFG via Research Unit FOR 2401, project 424107692 and 525018088, and by the EU via ELO-X 953348.
		}}%

\begin{document}

\maketitle
\thispagestyle{empty}
\pagestyle{empty}

\begin{abstract}
This paper extends the Finite Elements with Switch Detection and Jumps (FESD-J) \cite{fesdj} method to problems of rigid body dynamics involving patch contacts. The \mbox{FESD-J} method is a high accuracy discretization scheme suitable for use in direct optimal control of nonsmooth mechanical systems. It detects dynamic switches exactly in time and, thereby, maintains the integration order of the underlying Runge-Kutta (RK) method. This is in contrast to commonly used time-stepping methods which only achieve first-order accuracy. Considering rigid bodies with possible patch contacts results in nondifferentiable signed distance functions (SDF), which introduces additional nonsmoothness into the dynamical system. In this work, we utilize so-called equivalent contact points (ECP), which parameterize force and impulse distributions on contact patches by evaluation at single points. We embed a nondifferentiable SDF into a complementarity Lagrangian system (CLS) and show that the determined ECP are well-defined. We then extend the FESD-J discretization to the considered CLS such that its integration accuracy is maintained. The functionality of the method is illustrated for both a simulation and an optimal control example.
\end{abstract}

\section{INTRODUCTION}
The use of optimal control methods is advantageous for many robotic tasks as system dynamics and constraints can be easily integrated and resulting control inputs fulfill predefined optimality criteria. Applications such as locomotion and manipulation can be captured by complementarity \mbox{Lagrangian} systems (CLS) with state jumps \cite{brogliato2016}, which consider standard Newton-Euler equations of rigid body dynamics together with complementarity conditions for force interactions and impact laws. Since impacts and contact forces result in jumps and kinks of system velocities, these CLS fall into the class of nonsmooth dynamical systems. This makes using standard methods applied in direct optimal control difficult, as integration of dynamical systems is usually accomplished by utilizing high-order integration methods such as Runge-Kutta (RK) schemes across fixed time intervals, which for nonsmooth systems only achieves first-order accuracy \cite{acary}. A large part of the existing work on direct optimal control for nonsmooth systems relies on \mbox{(semi-)implicit} Euler time-stepping methods \cite{posa, xie}. These methods do not aim to detect switches in the system exactly but rather integrate forces and impulses simultaneously over fixed time intervals. Thereby, they produce wrong numerical sensitivities and introduce artificial minima in which the optimizer can get stuck \cite{stewart, zhong22}.

Recently, the Finite Elements with Switch Detection and Jumps (FESD-J) \cite{fesdj} method was introduced, which is an event-based discretization scheme for nonsmooth systems suitable for use in direct optimal control. FESD-J allows to detect switches in the dynamical system and state jumps exactly in time. Thereby, the integration order of an underlying RK method is recovered. The method was introduced, however, for the case where signed distance functions (SDF) are sufficiently smooth.

When considering problems of rigid body dynamics originating from real-world applications, in many cases objects interacting with each other admit to contact surfaces that are not singletons. A patch contact occurs if there are multiple distinct points that are simultaneously contained in the surfaces of two considered objects. To derive rigid body dynamic models of such systems one has to consider SDF between the involved objects, which are for the case of possible patch contacts nondifferentiable. This introduces additional nonsmoothness into the system and thereby provides further challenges for numerical treatment. In \cite{xie} a nondifferentiable SDF was integrated in a time-stepping method that allows to simulate the dynamical evolution of polytopic bodies making contact. The concept of an equivalent contact point (ECP) was introduced, which is a unique point on the contact patch such that force acting on this point is equivalent to a force distribution acting across the full contact patch. It was shown that for given discrete time position and velocity solutions of the time-stepping discretization, the corresponding ECP and force magnitudes are unique.

If one uses an event-based discretization method, impulses across patch contacts are simulated in an isolated manner. This is contrary to time-stepping methods as used \mbox{in \cite{xie}}, which treat impulses and forces simultaneously. Since impulses are complex microscopic phenomena that heavily depend on material properties, in the field of optimal control one often assumes that involved bodies are perfectly rigid and then uses a phenomenological model for impact \mbox{resolution \cite{brogliato2016}}. When dealing with patch contacts, one has to resolve multiple simultaneous impacts at once, as the velocity of all points on the contact patch needs to undergo a jump to prevent penetration of bodies in the next time instant. Treating simultaneous impacts requires additional modeling \mbox{choices \cite{nguyen}}.

\subsection{Contribution}
In this work, we embed a nondifferentiable SDF modeled as a convex optimization problem into a CLS with state jumps to obtain a continuous time model for rigid body dynamics of objects with possible patch contacts. In particular, we focus on the frictionless case and consider padded polytopes, which are obtained by sweeping standard polytopes with a sphere of arbitrarily small radius. This modeling decision is taken to avoid impacts between sharp corners, which require special treatment \cite{glocker2006}. We then propose a multi-impact law based on the considered SDF and show that for polytopic contact patches this law is equivalent to Moreau's impact law \cite{moreau}. We argue that the CLS is well-defined, i.e., that it produces a unique evolution of state trajectories (Section \ref{seccls}).

The FESD-J method is then extended to our model, such that the integration order is maintained and the considered impact law is reflected in the discretization. To this end, we have to introduce additional so-called cross complementarity conditions to detect switches in the nondifferentiable SDF exactly in time (Section \ref{secfesd}).

The functionality of the model and discretization is illustrated in two numerical examples, namely the simulation of a falling cuboid making contact with its environment and the optimal control of a planar manipulation task (Section \ref{secnumex}).

\subsection{Notation}
Left- and right-hand side limits of a function at some evaluation time $t_{\mathrm{s}}$ are denoted by $x(t_{\mathrm{s}}^{+}) = \displaystyle\lim_{t\rightarrow t_{\mathrm{s}}, t> t_{\mathrm{s}}} x(t)$ and $x(t_{\mathrm{s}}^{-}) = \displaystyle\lim_{t\rightarrow t_{\mathrm{s}}, t< t_{\mathrm{s}}} x(t)$, respectively. For ease of notation, we drop time dependencies of differential and algebraic variables when it is clear from the context. The complementarity condition for two vectors $a, b\in \mathbb{R}^{n}$ is denoted by \mbox{$ 0 \leq a  \perp  b \geq 0$}, where $ a  \perp  b$ means $ a ^{\top} b = 0$. We denote the rotation matrix in two dimensions for some $\theta \in \mathbb{R}$ by
\begin{equation*}
	R(\theta) = \begin{pmatrix}
		\cos(\theta) & -\sin{\theta} \\
		\sin(\theta) & \cos{\theta}
	\end{pmatrix}.
\end{equation*}
The cross product for two vectors $a, b\in \mathbb{R}^{2}$ is given by
$a \times b = a^{\top} R\left(-\frac{\pi}{2}\right)b = a_1 b_2 - a_2b_1.$ Given a vector $a \in \mathbb{R}^{n}$ and a set of indices $\mathcal{I} \subset \{1,\dots,n\}$, $a_{\mathcal{I}} \in \mathbb{R}^{\vert\mathcal{I} \vert}$ is the vector only containing the components of $a$ with indices in $\mathcal{I}$.
\section{RIGID BODY DYNAMICS WITH PATCH CONTACTS} \label{seccls}

We consider CLS with state jumps that arise in nonsmooth mechanics for rigid bodies with inelastic impacts. In particular, we consider the case where these bodies are modeled by padded polytopes. This is a challenging problem since SDF between such objects are not everywhere differentiable, which introduces additional nonsmoothness into the system.

In the following, we first express a SDF between padded polytopes as an optimization problem and discuss properties of resulting primal and dual variables in Section \ref{section1a}. In Section \ref{section1b}, we state the considered CLS and discuss how to obtain contact normal vectors that specify the direction along which forces and impulses act on the system. We then utilize the SDF to formulate a simultaneous impact law and relate our formulation to Moreau's impact law \cite{moreau} in Section \ref{section1imp}. Finally, in Section \ref{section1unique}, we show that the considered CLS produces a unique evolution of state trajectories. In particular, if objects make patch contacts, there is a unique ECP chosen to parameterize the contact force interaction. 

\begin{figure*}
	\centering
	\begin{subfigure}{.32\linewidth} 
		\centering 
\begin{tikzpicture}


\definecolor{darkgray176}{RGB}{176,176,176}

\begin{axis}[
width=0.7\textwidth,
height=0.466\textwidth,
scale only axis=true,
tick align=outside,
tick pos=left,
x grid style={darkgray176},
xmin=-3, xmax=3,
xtick style={color=black},
y grid style={darkgray176},
ymin=-2, ymax=2,
ytick style={color=black},
ylabel= {\(\displaystyle c_{\mathrm{y}}\)},
xlabel= {\(\displaystyle c_{\mathrm{x}}\)},
x label style={at={(axis description cs:0.5,-0.17)},anchor=north},
y label style={at={(axis description cs:-0.12,.5)}, anchor=south},
label style={font=\scriptsize},
tick label style={font=\scriptsize},
]
\path [draw=black, fill=mediumorange]
(axis cs:2.8,-0.3)
--(axis cs:2.8,-1.5)
.. controls (axis cs:2.8,-1.76666666666667) and (axis cs:2.66666666666667,-1.9) .. (axis cs:2.4,-1.9)
--(axis cs:1.2,-1.9)
.. controls (axis cs:0.933333333333333,-1.9) and (axis cs:0.8,-1.76666666666667) .. (axis cs:0.8,-1.5)
--(axis cs:0.8,-0.3)
.. controls (axis cs:0.8,-0.0333333333333335) and (axis cs:0.933333333333333,0.0999999999999999) .. (axis cs:1.2,0.0999999999999999)
--(axis cs:2.4,0.1)
.. controls (axis cs:2.66666666666667,0.1) and (axis cs:2.8,-0.0333333333333334) .. (axis cs:2.8,-0.3);
\path [draw=black, fill=mediumcyan]
(axis cs:-2.4,1.9)
--(axis cs:-1.2,1.9)
.. controls (axis cs:-0.933333333333333,1.9) and (axis cs:-0.8,1.76666666666667) .. (axis cs:-0.8,1.5)
--(axis cs:-0.8,0.3)
.. controls (axis cs:-0.8,0.0333333333333334) and (axis cs:-0.933333333333333,-0.1) .. (axis cs:-1.2,-0.1)
--(axis cs:-2.4,-0.1)
.. controls (axis cs:-2.66666666666667,-0.1) and (axis cs:-2.8,0.0333333333333334) .. (axis cs:-2.8,0.3)
--(axis cs:-2.8,1.5)
.. controls (axis cs:-2.8,1.76666666666667) and (axis cs:-2.66666666666667,1.9) .. (axis cs:-2.4,1.9);
\path [draw=black, dotted]
(axis cs:-1.2,1.5)
--(axis cs:-1.2,0.3)
--(axis cs:-2.4,0.3)
--(axis cs:-2.4,1.5)
--cycle;
\path [draw=black, dotted]
(axis cs:2.4,-1.5)
--(axis cs:1.2,-1.5)
--(axis cs:1.2,-0.3)
--(axis cs:2.4,-0.3)
--cycle;
\draw[draw=black, dashed] (axis cs:1.2,-0.3) circle (0.4);
\draw[draw=black, dashed] (axis cs:-1.2,0.3) circle (0.4);
\path [draw=black, fill=mediumorange, opacity=0.4]
(axis cs:3.60000003008182,0.500000030081823)
--(axis cs:3.60000003008182,-2.30000003008182)
.. controls (axis cs:3.60000003008182,-2.56666669674849) and (axis cs:3.46666669674849,-2.70000003008182) .. (axis cs:3.20000003008182,-2.70000003008182)
--(axis cs:0.399999969918177,-2.70000003008182)
.. controls (axis cs:0.133333303251511,-2.70000003008182) and (axis cs:-3.00818225884569e-08,-2.56666669674849) .. (axis cs:-3.00818225884569e-08,-2.30000003008182)
--(axis cs:-3.00818225884569e-08,0.500000030081823)
.. controls (axis cs:-3.00818225884569e-08,0.766666696748489) and (axis cs:0.133333303251511,0.900000030081822) .. (axis cs:0.399999969918177,0.900000030081823)
--(axis cs:3.20000003008182,0.900000030081823)
.. controls (axis cs:3.46666669674849,0.900000030081822) and (axis cs:3.60000003008182,0.766666696748489) .. (axis cs:3.60000003008182,0.500000030081823);
\path [draw=black, fill=mediumcyan, opacity=0.4]
(axis cs:-3.20000003008182,2.70000003008182)
--(axis cs:-0.399999969918177,2.70000003008182)
.. controls (axis cs:-0.133333303251511,2.70000003008182) and (axis cs:3.00818225884569e-08,2.56666669674849) .. (axis cs:3.00818225884569e-08,2.30000003008182)
--(axis cs:3.00818225884569e-08,-0.500000030081823)
.. controls (axis cs:3.00818225884569e-08,-0.766666696748489) and (axis cs:-0.133333303251511,-0.900000030081822) .. (axis cs:-0.399999969918177,-0.900000030081823)
--(axis cs:-3.20000003008182,-0.900000030081823)
.. controls (axis cs:-3.46666669674849,-0.900000030081822) and (axis cs:-3.60000003008182,-0.766666696748489) .. (axis cs:-3.60000003008182,-0.500000030081823)
--(axis cs:-3.60000003008182,2.30000003008182)
.. controls (axis cs:-3.60000003008182,2.56666669674849) and (axis cs:-3.46666669674849,2.70000003008182) .. (axis cs:-3.20000003008182,2.70000003008182);
\path [draw=black, draw opacity=0.4, dotted]
(axis cs:3.20000003008182,-2.30000003008182)
--(axis cs:0.399999969918177,-2.30000003008182)
--(axis cs:0.399999969918177,0.500000030081823)
--(axis cs:3.20000003008182,0.500000030081823)
--cycle;
\path [draw=black, draw opacity=0.4, dotted]
(axis cs:-0.399999969918177,2.30000003008182)
--(axis cs:-0.399999969918177,-0.500000030081823)
--(axis cs:-3.20000003008182,-0.500000030081823)
--(axis cs:-3.20000003008182,2.30000003008182)
--cycle;
\draw[draw=black,draw opacity=0.4, dashed] (axis cs:0.399999969918177,0.500000030081823) circle (0.4);
\draw[draw=black,draw opacity=0.4, dashed] (axis cs:-0.399999969918177,-0.500000030081823) circle (0.4);

\draw[semithick, line width=1.5pt, mediumgreen,arrows={-Triangle[scale=0.8]}]
(0,-1.5) -- (-0.8,-1.5);
\draw (axis cs:-0.97,-1.18) node[
  anchor= west,
  text=mediumgreen,
  rotate=0.0
]{$n_{\mathrm{tr}}$};

\draw [semithick, mediumgreen, line width=1.5pt, dashed] (0,3) -- (0,-3);

\addplot [semithick, black, mark=*, mark size=1.5, mark options={solid}]
table {%
1.8 -0.9
};

\addplot [semithick, black, mark=*, mark size=1.5, mark options={solid}]
table {%
-1.8 0.9
};

\draw (axis cs:-1.8, 0.9) node[
  anchor=south,
  text=black,
  rotate=0.0
]{$c_1$};

\draw (axis cs:1.8, -0.9) node[
  anchor=south,
  text=black,
  rotate=0.0
]{$c_2$};

\addplot [semithick, mediumred, line width=1.5pt]
table {%
0.399999969918177 -0.500000030081823
0.399999969918177 0.500000030081823
};
\addplot [semithick, mediumred, line width=1.5pt]
table {%
-0.399999969918177 -0.500000030081823
-0.399999969918177 0.500000030081823
};
\addplot [semithick, mediumblue, line width=1.5pt]
table {%
3.00818226994792e-08 -0.500000030081823
3.00818226994792e-08 0.500000030081823
};
\addplot [semithick, mediumblue, mark=*, mark size=3, mark options={solid}]
table {%
-4.67332237246747e-17 0.200000000278001
};
\addplot [semithick, mediumred, mark=*, mark size=3, mark options={solid}]
table {%
0.399999986209088 0.199999983817384
};
\addplot [semithick, mediumred, mark=*, mark size=3, mark options={solid}]
table {%
-0.399999986209088 0.200000016738618
};
\draw (axis cs:-0.15,0.15) node[
  anchor=north west,
  text=mediumblue,
  rotate=0.0
]{$p$};
\draw (axis cs:0.28,0.95) node[
  anchor=north west,
  text=mediumred,
  rotate=0.0
]{$y_2$};
\draw (axis cs:-0.14,0.34) node[
  anchor=south east,
  text=mediumred,
  rotate=0.0
]{$y_1$};
\end{axis}

\end{tikzpicture}
		\vspace{-0.5em}
		\caption{}
		\label{distfig:a}
	\end{subfigure}
	\hfill
	\begin{subfigure}{.32\linewidth}
		\centering 	
\begin{tikzpicture}

\definecolor{crimson2143940}{RGB}{214,39,40}
\definecolor{darkgray176}{RGB}{176,176,176}
\definecolor{darkorange25512714}{RGB}{255,140,30}

\begin{axis}[
width=0.7\textwidth,
height=0.466\textwidth,
scale only axis=true,
tick align=outside,
tick pos=left,
x grid style={darkgray176},
xmin=-3, xmax=3,
xtick style={color=black},
y grid style={darkgray176},
ymin=-2, ymax=2,
ytick style={color=black},
ylabel= {\(\displaystyle c_{\mathrm{y}}\)},
xlabel= {\(\displaystyle c_{\mathrm{x}}\)},
x label style={at={(axis description cs:0.5,-0.17)},anchor=north},
y label style={at={(axis description cs:-0.12,.5)}, anchor=south},
label style={font=\scriptsize},
tick label style={font=\scriptsize},
]
\path [draw=black, fill=mediumorange]
(axis cs:2.61936922882972,-0.0701602161579664)
--(axis cs:2.91779709342754,-1.23246000951232)
.. controls (axis cs:2.9841143966715,-1.49074885247996) and (axis cs:2.88812862680967,-1.65305192558576) .. (axis cs:2.62983978384203,-1.71936922882972)
--(axis cs:1.46753999048768,-2.01779709342754)
.. controls (axis cs:1.20925114752004,-2.08411439667151) and (axis cs:1.04694807441424,-1.98812862680967) .. (axis cs:0.980630771170282,-1.72983978384203)
--(axis cs:0.682202906572456,-0.567539990487676)
.. controls (axis cs:0.615885603328495,-0.309251147520041) and (axis cs:0.711871373190331,-0.146948074414243) .. (axis cs:0.970160216157966,-0.0806307711702817)
--(axis cs:2.13246000951232,0.217797093427544)
.. controls (axis cs:2.39074885247996,0.284114396671506) and (axis cs:2.55305192558576,0.188128626809669) .. (axis cs:2.61936922882972,-0.0701602161579664);
\path [draw=black, fill=mediumcyan]
(axis cs:-2.4,1.9)
--(axis cs:-1.2,1.9)
.. controls (axis cs:-0.933333333333333,1.9) and (axis cs:-0.8,1.76666666666667) .. (axis cs:-0.8,1.5)
--(axis cs:-0.8,0.3)
.. controls (axis cs:-0.8,0.0333333333333334) and (axis cs:-0.933333333333333,-0.1) .. (axis cs:-1.2,-0.1)
--(axis cs:-2.4,-0.1)
.. controls (axis cs:-2.66666666666667,-0.1) and (axis cs:-2.8,0.0333333333333334) .. (axis cs:-2.8,0.3)
--(axis cs:-2.8,1.5)
.. controls (axis cs:-2.8,1.76666666666667) and (axis cs:-2.66666666666667,1.9) .. (axis cs:-2.4,1.9);
\path [draw=black, dotted]
(axis cs:-1.2,1.5)
--(axis cs:-1.2,0.3)
--(axis cs:-2.4,0.3)
--(axis cs:-2.4,1.5)
--cycle;
\path [draw=black, dotted]
(axis cs:2.53036382897609,-1.33193596437827)
--(axis cs:1.36806403562173,-1.63036382897609)
--(axis cs:1.06963617102391,-0.468064035621734)
--(axis cs:2.23193596437827,-0.169636171023908)
--cycle;
\draw[draw=black, dashed] (axis cs:1.06963617102391,-0.468064035621734) circle (0.4);
\draw[draw=black, dashed] (axis cs:-1.2,0.3) circle (0.4);

\draw[semithick, line width=1.5pt, mediumgreen,arrows={-Triangle[scale=0.8]}]
(-0.137187546675348,-1.5) -- (-0.93718,-1.5);
\draw (axis cs:-1.107,-1.18) node[
  anchor= west,
  text=mediumgreen,
  rotate=0.0
]{$n_{\mathrm{tr}}$};

\draw [semithick, mediumgreen, line width=1.5pt, dashed] (-0.137187546675348,3) -- (-0.137187546675348,-3);

\addplot [semithick, black, mark=*, mark size=1.5, mark options={solid}]
table {%
1.8 -0.9
};

\addplot [semithick, black, mark=*, mark size=1.5, mark options={solid}]
table {%
-1.8 0.9
};

\draw (axis cs:-1.8, 0.9) node[
  anchor=south,
  text=black,
  rotate=0.0
]{$c_1$};

\draw (axis cs:1.8, -0.9) node[
  anchor=south,
  text=black,
  rotate=0.0
]{$c_2$};

\path [draw=black, fill=mediumorange, opacity=0.4]
(axis cs:3.09652347826587,0.736663557203977)
--(axis cs:3.72462086678949,-1.70961425894848)
.. controls (axis cs:3.79093817003345,-1.96790310191611) and (axis cs:3.69495240017161,-2.13020617502191) .. (axis cs:3.43666355720398,-2.19652347826587)
--(axis cs:0.990385741051522,-2.82462086678949)
.. controls (axis cs:0.732096898083887,-2.89093817003345) and (axis cs:0.569793824978089,-2.79495240017161) .. (axis cs:0.503476521734128,-2.53666355720398)
--(axis cs:-0.124620866789488,-0.090385741051522)
.. controls (axis cs:-0.190938170033449,0.167903101916113) and (axis cs:-0.0949524001716121,0.330206175021911) .. (axis cs:0.163336442796023,0.396523478265872)
--(axis cs:2.60961425894848,1.02462086678949)
.. controls (axis cs:2.86790310191611,1.09093817003345) and (axis cs:3.03020617502191,0.994952400171612) .. (axis cs:3.09652347826587,0.736663557203977);
\path [draw=black, fill=mediumcyan, opacity=0.4]
(axis cs:-3.06281248442413,2.56281248442413)
--(axis cs:-0.537187515575867,2.56281248442413)
.. controls (axis cs:-0.2705208489092,2.56281248442413) and (axis cs:-0.137187515575867,2.4294791510908) .. (axis cs:-0.137187515575867,2.16281248442413)
--(axis cs:-0.137187515575867,-0.362812484424133)
.. controls (axis cs:-0.137187515575867,-0.6294791510908) and (axis cs:-0.2705208489092,-0.762812484424133) .. (axis cs:-0.537187515575867,-0.762812484424133)
--(axis cs:-3.06281248442413,-0.762812484424133)
.. controls (axis cs:-3.3294791510908,-0.762812484424133) and (axis cs:-3.46281248442413,-0.6294791510908) .. (axis cs:-3.46281248442413,-0.362812484424133)
--(axis cs:-3.46281248442413,2.16281248442413)
.. controls (axis cs:-3.46281248442413,2.4294791510908) and (axis cs:-3.3294791510908,2.56281248442413) .. (axis cs:-3.06281248442413,2.56281248442413);
\path [draw=black, draw opacity=0.4, dotted]
(axis cs:3.33718760233803,-1.80909021381442)
--(axis cs:0.89090978618558,-2.43718760233804)
--(axis cs:0.262812397661965,0.00909021381441977)
--(axis cs:2.70909021381442,0.637187602338035)
--cycle;
\path [draw=black, draw opacity=0.4, dotted]
(axis cs:-0.537187515575867,2.16281248442413)
--(axis cs:-0.537187515575867,-0.362812484424133)
--(axis cs:-3.06281248442413,-0.362812484424133)
--(axis cs:-3.06281248442413,2.16281248442413)
--cycle;

\draw[draw=black,draw opacity=0.4, dashed] (axis cs:0.262812397661965,0.00909021381441977) circle (0.4);
\draw[draw=black,draw opacity=0.4, dashed] (axis cs:-0.537187515575867,-0.362812484424133) circle (0.4);

\addplot [semithick, mediumblue, mark=*, mark size=3, mark options={solid}]
table {%
-0.137187546675348 0.00909014757983628
};
\addplot [semithick, mediumred, mark=*, mark size=3, mark options={solid}]
table {%
0.262812439069638 0.00909012376001116
};
\addplot [semithick, mediumred, mark=*, mark size=3, mark options={solid}]
table {%
-0.537187532420334 0.0090901713996614
};

\draw (axis cs:-0.190,-0.1) node[
  anchor=north west,
  text=mediumblue,
  rotate=0.0
]{$p$};
\draw (axis cs:0.28,0.55) node[
  anchor=north west,
  text=mediumred,
  rotate=0.0
]{$y_2$};
\draw (axis cs:-0.10,0.105) node[
  anchor=south east,
  text=mediumred,
  rotate=0.0
]{$y_1$};
\end{axis}

\end{tikzpicture}
		\vspace{-0.5em}
		\caption{}
		\label{distfig:b}
	\end{subfigure}
	\hfill
	\begin{subfigure}{.32\linewidth}
		\centering 	
\begin{tikzpicture}

\definecolor{crimson2143940}{RGB}{214,39,40}
\definecolor{darkgray176}{RGB}{176,176,176}
\definecolor{darkorange25512714}{RGB}{255,127,14}
\definecolor{steelblue31119180}{RGB}{31,119,180}

\begin{axis}[
width=0.7\textwidth,
height=0.466\textwidth,
scale only axis=true,
tick align=outside,
tick pos=left,
x grid style={darkgray176},
xmin=-3, xmax=3,
xtick style={color=black},
y grid style={darkgray176},
ymin=-2, ymax=2,
ytick style={color=black},
ylabel= {\(\displaystyle c_{\mathrm{y}}\)},
xlabel= {\(\displaystyle c_{\mathrm{x}}\)},
x label style={at={(axis description cs:0.5,-0.17)},anchor=north},
y label style={at={(axis description cs:-0.12,.5)}, anchor=south},
label style={font=\scriptsize},
tick label style={font=\scriptsize},
]
\path [draw=black, fill=mediumorange]
(axis cs:2.32283184851462,0.142423550280205)
--(axis cs:2.96582400248941,-0.870769960322213)
.. controls (axis cs:3.10871114781715,-1.09592407378942) and (axis cs:3.06757766374741,-1.27994470318688) .. (axis cs:2.84242355028021,-1.42283184851462)
--(axis cs:1.82923003967779,-2.06582400248941)
.. controls (axis cs:1.60407592621058,-2.20871114781715) and (axis cs:1.42005529681312,-2.16757766374741) .. (axis cs:1.27716815148538,-1.94242355028021)
--(axis cs:0.634175997510587,-0.929230039677788)
.. controls (axis cs:0.491288852182855,-0.704075926210583) and (axis cs:0.53242233625259,-0.520055296813115) .. (axis cs:0.757576449719795,-0.377168151485383)
--(axis cs:1.77076996032221,0.265824002489413)
.. controls (axis cs:1.99592407378942,0.408711147817145) and (axis cs:2.17994470318688,0.36757766374741) .. (axis cs:2.32283184851462,0.142423550280205);
\path [draw=black, fill=mediumcyan]
(axis cs:-2.4,1.9)
--(axis cs:-1.2,1.9)
.. controls (axis cs:-0.933333333333333,1.9) and (axis cs:-0.8,1.76666666666667) .. (axis cs:-0.8,1.5)
--(axis cs:-0.8,0.3)
.. controls (axis cs:-0.8,0.0333333333333334) and (axis cs:-0.933333333333333,-0.1) .. (axis cs:-1.2,-0.1)
--(axis cs:-2.4,-0.1)
.. controls (axis cs:-2.66666666666667,-0.1) and (axis cs:-2.8,0.0333333333333334) .. (axis cs:-2.8,0.3)
--(axis cs:-2.8,1.5)
.. controls (axis cs:-2.8,1.76666666666667) and (axis cs:-2.66666666666667,1.9) .. (axis cs:-2.4,1.9);
\path [draw=black, dotted]
(axis cs:-1.2,1.5)
--(axis cs:-1.2,0.3)
--(axis cs:-2.4,0.3)
--(axis cs:-2.4,1.5)
--cycle;
\path [draw=black, dotted]
(axis cs:2.62809283228861,-1.08510067831381)
--(axis cs:1.61489932168619,-1.72809283228861)
--(axis cs:0.971907167711393,-0.714899321686189)
--(axis cs:1.98510067831381,-0.071907167711393)
--cycle;
\draw[draw=black, dashed] (axis cs:0.971907167711393,-0.714899321686189) circle (0.4);
\draw[draw=black, dashed] (axis cs:-1.2,0.3) circle (0.4);
\path [draw=black, fill=mediumorange, opacity=0.4]
(axis cs:2.50544112358974,0.959370485195065)
--(axis cs:3.78277093740427,-1.05337923539734)
.. controls (axis cs:3.92565808273201,-1.27853334886454) and (axis cs:3.88452459866227,-1.46255397826201) .. (axis cs:3.65937048519507,-1.60544112358974)
--(axis cs:1.64662076460266,-2.88277093740427)
.. controls (axis cs:1.42146665113546,-3.025658082732) and (axis cs:1.23744602173799,-2.98452459866227) .. (axis cs:1.09455887641026,-2.75937048519507)
--(axis cs:-0.182770937404273,-0.746620764602663)
.. controls (axis cs:-0.325658082732005,-0.521466651135459) and (axis cs:-0.284524598662269,-0.33744602173799) .. (axis cs:-0.0593704851950652,-0.194558876410258)
--(axis cs:1.95337923539734,1.08277093740427)
.. controls (axis cs:2.17853334886454,1.22565808273201) and (axis cs:2.36255397826201,1.18452459866227) .. (axis cs:2.50544112358974,0.959370485195065);
\path [draw=black, fill=mediumcyan, opacity=0.4]
(axis cs:-2.99192416820495,2.49192416820495)
--(axis cs:-0.608075831795049,2.49192416820495)
.. controls (axis cs:-0.341409165128382,2.49192416820495) and (axis cs:-0.208075831795049,2.35859083487162) .. (axis cs:-0.208075831795049,2.09192416820495)
--(axis cs:-0.208075831795049,-0.291924168204951)
.. controls (axis cs:-0.208075831795049,-0.558590834871618) and (axis cs:-0.341409165128382,-0.691924168204951) .. (axis cs:-0.608075831795049,-0.691924168204951)
--(axis cs:-2.99192416820495,-0.691924168204951)
.. controls (axis cs:-3.25859083487162,-0.691924168204951) and (axis cs:-3.39192416820495,-0.558590834871618) .. (axis cs:-3.39192416820495,-0.291924168204951)
--(axis cs:-3.39192416820495,2.09192416820495)
.. controls (axis cs:-3.39192416820495,2.35859083487162) and (axis cs:-3.25859083487162,2.49192416820495) .. (axis cs:-2.99192416820495,2.49192416820495);
\path [draw=black, draw opacity=0.4, dotted]
(axis cs:3.44503976720347,-1.26770995338894)
--(axis cs:1.43229004661106,-2.54503976720347)
--(axis cs:0.154960232796534,-0.532290046611064)
--(axis cs:2.16770995338894,0.745039767203467)
--cycle;
\path [draw=black, draw opacity=0.4, dotted]
(axis cs:-0.608075831795049,2.09192416820495)
--(axis cs:-0.608075831795049,-0.291924168204951)
--(axis cs:-2.99192416820495,-0.291924168204951)
--(axis cs:-2.99192416820495,2.09192416820495)
--cycle;
\draw[draw=black,draw opacity=0.4, dashed] (axis cs:0.154960232796534,-0.532290046611064) circle (0.4);
\draw[draw=black,draw opacity=0.4, dashed] (axis cs:-0.608075831795049,-0.291924168204951) circle (0.4);

\draw[semithick, line width=1.5pt, mediumgreen,arrows={-Triangle[scale=0.8]}]
(-0.5480472917441549,-1.4326679388972259) -- (-1.3110834594604128,-1.192301953872541);
\draw (axis cs:-1.257,-1.02) node[
  anchor= west,
  text=mediumgreen,
  rotate=0.0
]{$n_{\mathrm{tr}}$};

\draw [semithick, mediumgreen, line width=1.5pt, dashed] (0.97527214, 3.40307377) -- (-1.42838771, -4.2272879);

\addplot [semithick, black, mark=*, mark size=1.5, mark options={solid}]
table {%
1.8 -0.9
};

\addplot [semithick, black, mark=*, mark size=1.5, mark options={solid}]
table {%
-1.8 0.9
};

\draw (axis cs:-1.8, 0.9) node[
  anchor=south,
  text=black,
  rotate=0.0
]{$c_1$};

\draw (axis cs:1.8, -0.9) node[
  anchor=south,
  text=black,
  rotate=0.0
]{$c_2$};

\addplot [semithick, mediumblue, mark=*, mark size=3, mark options={solid}]
table {%
-0.226557786773639 -0.412107064576731
};
\addplot [semithick, mediumred, mark=*, mark size=3, mark options={solid}]
table {%
0.154960280154066 -0.532290051755777
};
\addplot [semithick, mediumred, mark=*, mark size=3, mark options={solid}]
table {%
-0.608075853701344 -0.291924077397685
};
\draw (axis cs:-0.45,-0.49) node[
  anchor=north west,
  text=mediumblue,
  rotate=0.0
]{$p$};
\draw (axis cs:0.04,0.14) node[
  anchor=north west,
  text=mediumred,
  rotate=0.0
]{$y_2$};
\draw (axis cs:-0.13,-0.22) node[
  anchor=south east,
  text=mediumred,
  rotate=0.0
]{$y_1$};
\end{axis}

\end{tikzpicture}
		\vspace{-0.5em}
		\caption{}
		\label{distfig:c}
	\end{subfigure}
	\caption{Visualization of two unscaled padded polytopes in bold color, their downsized interior polytopes (dotted) and exemplary circles used for padding (dashed). Further, scaled versions of the padded polytopes according to the scaling factor $\alpha$ determined by the distance problem (\ref{distprob}) are visualized in light color.}
	\vspace{-1em}
	\label{distfig}
\end{figure*}
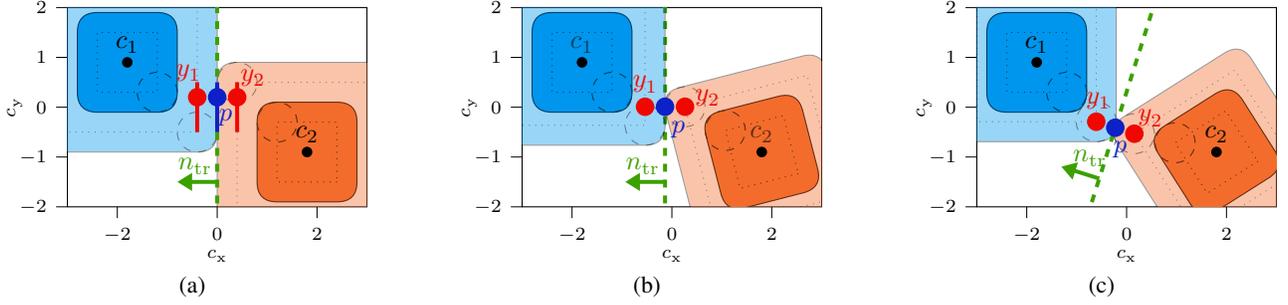

\subsection{Signed Distance Function for Padded Polytopes} \label{section1a}
Consider two planar polytopic bodies given in a halfspace representation
\begin{equation*}
\mathcal{P}_{1} = \{p \in \mathbb{R}^{2}\ \vert \  A_{1}x \leq b_{1}\}, \ \mathcal{P}_{2} = \{p \in \mathbb{R}^{2} \ \vert \  A_{2}x\leq b_{2}\},
\end{equation*}
where $A_{1} \in \mathbb{R}^{\nha \times 2}$, $b_{1} \in \mathbb{R}^{\nha}$, $A_{2} \in \mathbb{R}^{\nhb \times 2}$, \mbox{$b_{2} \in \mathbb{R}^{\nhb}$}. Without loss of generality, assume that for a given mass distribution, the halfspace representation of the polytopes is constructed such that their center of mass (CoM) coincides with the origin. Furthermore, let the rows of $A_{1}$ and $A_{2}$ be vectors with norm unity. We also assume that for a scalar $r>0$, it holds $b_{i} - \bm{1}r > 0$, $i=1,2$, where $\bm{1} = (1,\dots,1)^{\top}$ denotes the vector of all ones. This means that the minimum Euclidean distance from the CoM to the boundary of the respective polytopes is larger than $r$. 

We now transform the polytopes by several alterations. First we associate degrees of freedom (DoF) with them, namely $c_{i}\in \mathbb{R}^{2}$ for the position of their CoM and $\xi_{i} \in \mathbb{R}$ for their orientation, both with respect to a world coordinate frame. We capture all DoF of the bodies in a system configuration variable $q = (c_1,\xi_1,c_2,\xi_2)$. Next, we replace the right-hand side vectors $b_i$ by $b_i - \bm{1}r$ to obtain downsized polytopes, which are then swept by a sphere of radius $r$. This yields an inner approximation of the original polytopes with round corners. Finally, we introduce the scaling factor $\alpha > 0$, which allows one to obtain scaled versions of the padded polytopes by altering the size of the downsized polytopes, for $\alpha = 1$ the original size is retained. We obtain analytic expressions for the scaled padded polytopes by
\begin{alignat*}{1}
	\mathcal{B}_{1}(q,\alpha) &= \{p \in \mathbb{R}^{2} \mid \exists \, y_1  \in \mathbb{R}^{2}: \\
	&A_{1}R( \xi_{1})^{\top}(y_1-c_{1}) \leq \alpha b_{1} - \bm{1}r, \Vert p - y_{1} \Vert_{2} \leq r\}, \\
	\mathcal{B}_{2}(q,\alpha) &= \{p \in \mathbb{R}^{2} \mid \exists \, y_2  \in \mathbb{R}^{2}: \\
	&A_{2}R( \xi_{2})^{\top}(y_2-c_{2}) \leq \alpha b_{2} - \bm{1}r, \Vert p - y_{2} \Vert_{2} \leq r\}.
\end{alignat*}
See Fig. \ref{distfig} for a visualization of these sets.

A SDF $\Phi(q)$ for the padded polytopes $\mathcal{B}_{1}(q,1)$ and $\mathcal{B}_{2}(q,1)$  can be expressed by determining the smallest scaling factor such that there is a common point in both scaled bodies \cite{dcol}. Using the analytic representations of the sets derived above, we formulate this optimization problem by
\begin{minie}
	{p,\alpha, y_{1},y_{2}}{\alpha-1\phantom{aaaaaaaaaaaaaaa}\label{distprob:a}}{\label{distprob}}{\Phi(q) =}
	\addConstraint{A_{1}R( \xi_{1})^{\top}(y_{1}-c_{1})}{\leq \alpha b_{1} - \bm{1}r, \label{distprob:b}}
	\addConstraint{A_{2}R( \xi_{2})^{\top}(y_{2}-c_{2})}{\leq \alpha b_{2} - \bm{1}r, \label{distprob:c}}
	\addConstraint{(p-y_{1})^{\top}(p-y_{1})}{\leq  r^2, \label{distprob:d}}
	\addConstraint{(p-y_{2})^{\top}(p-y_{2})}{\leq r^2, \label{distprob:e}}
\end{minie}
which we write compactly as
\begin{equation} \label{distprob_compact}
\Phi(q) =\min_{z} \ f_{\mathrm{d}}(z) \quad \mathrm{s.t.} \quad g_{\mathrm{d}}(z,q)\leq 0,
\end{equation}
where $z = (p,\alpha, y_{1},y_{2}) \in \mathbb{R}^{7}$ collects all primal variables. 

The convex optimization problem (\ref{distprob}) fulfills Slater's constraint qualification, as there always exists a feasible point for which the constraints are not active. This can be seen from Fig. \ref{distfig} as by choosing a larger $\alpha$ the scaled padded polytopes will overlap and, thus, $p$, $y_1$, $y_2$, can be chosen such that the constraints of (\ref{distprob}) are fulfilled with strict inequalities. Consequently, the Karush-Kuhn-Tucker (KKT) conditions are necessary and sufficient for optimality.

In the following, we will investigate how optimal primal and dual solutions to (\ref{distprob}) depend on the system \mbox{configuration $q$}. To this end, we associate nonnegative Lagrange multipliers $\mu_{1, \mathrm{p}} \in \mathbb{R}^{\nha}$, $\mu_{2, \mathrm{p}}\in \mathbb{R}^{\nhb}$, \mbox{$\mu_{1, \mathrm{c}}\in \mathbb{R}$}, $\mu_{2, \mathrm{c}}\in \mathbb{R}$, with the constraints (\ref{distprob:b})-(\ref{distprob:e}) and collect them in the vector $\mu = (\mu_{1, \mathrm{p}},\mu_{2, \mathrm{p}},\mu_{1, \mathrm{c}},\mu_{2, \mathrm{c}})$. The stationary KKT condition of (\ref{distprob}), i.e., $\nabla_{z} f_{\mathrm{d}}(z) + \nabla_{z} g_{\mathrm{d}}(z,q)\mu = 0$, evaluates as the system of equations given by
\begin{equation} \label{statKKTcond}
	\begin{split}
		1 -b_{1}^{\top} \mu_{1, \mathrm{p}} -b_{2}^{\top} \mu_{2, \mathrm{p}} &= 0, \\
		2(p-y_{1})\mu_{1, \mathrm{c}} &= - 2(p-y_{2})\mu_{2, \mathrm{c}},  \\
		R( \xi_{1})A_{1}^{\top}\mu_{1, \mathrm{p}} &=  2(p-y_{1}) \mu_{1, \mathrm{c}}, \\
		R( \xi_{2})A_{2}^{\top}\mu_{2, \mathrm{p}} &=  2(p-y_{2}) \mu_{2, \mathrm{c}}.
	\end{split}
\end{equation}
The following proposition holds true for problem (\ref{distprob}).

\begin{proposition} \label{uniquepdvar}
For any  $q \in \mathbb{R}^{n_{q}}$, apart from the degenerate case $c_1 = c_2$, optimal dual variables for (\ref{distprob}) are unique and depend continuously on $q$. Optimal primal variables are unique if there is no patch contact between the scaled padded polytopes.
\end{proposition}

A full proof can be found in Appendix \ref{app:proof1}. In the following, we briefly give some intuition on why this holds true with the help of the visualizations in Fig. \ref{distfig}. First, we observe that the optimal scaling factor $\alpha$ is uniquely determined as increasing it violates optimality and if it is decreased, no common point $p$ is included in both scaled bodies. In the case of Fig. \ref{distfig:b} or Fig. \ref{distfig:c}, we observe that there is a unique point $p$ that is element of both scaled shapes, as well as unique points $y_1$ and $y_2$ that are elements of the downscaled polytopes. In the case of Fig. \ref{distfig:a}, we observe that there are multiple viable solutions for $p$, $y_1$ and $y_2$. However, we note that the vectors $p-y_{1}$ and $p-y_{2}$ are unique in any case of Fig. \ref{distfig}. Consequently, one can show that system (\ref{statKKTcond}) together with the complementarity condition of the KKT system is uniquely solvable for the Lagrange multipliers.

\subsection{Complementarity Lagrangian System} \label{section1b}
Dynamic interactions between rigid bodies are commonly modeled by CLS, which combine Newton-Euler equations with the complementarity constraint $0 \leq \Phi(q) \perp \lambda_{\mathrm{n}} \geq 0$, where $\lambda_{\mathrm{n}}$ is the contact force magnitude. If $\Phi(q)$ is differentiable, forces and impulses act along the SDF gradient, which is referred to as the contact normal vector. Since the SDF considered here is not differentiable for configurations as in Fig. \ref{distfig:a}, standard theory does not directly imply. However, we can determine directional derivatives by
\begin{equation} \label{dirderv}
\partial_{d}\Phi(q) = \min_{z, \mu \text{ solve } (\ref{distprob})} d^{\top} \nabla_q  g_{\mathrm{d}}(z,q) \mu ,
\end{equation}
for the direction $d\in \mathbb{R}^{n_{q}}$, cf. \cite{hogan}. Thus, we observe that $\Phi(q)$ is differentiable whenever there exists a unique primal solution to the distance problem (\ref{distprob}). Equation (\ref{dirderv}) motivates us to consider contact normals given by
\begin{align} 
	n(z,\mu,q) = \nabla_q  g_{\mathrm{d}}(z,q) \mu &= \begin{pmatrix}
		n_{\mathrm{tr}}(\mu,q) \\
		(p - c_{1}) \times 	n_{\mathrm{tr}}(\mu,q) \\
		-n_{\mathrm{tr}}(\mu,q) \\
		- (p - c_{2}) \times 	n_{\mathrm{tr}}(\mu,q)
	\end{pmatrix}, \label{normalform} \\
	n_{\mathrm{tr}}(\mu,q) &= - R(\xi_{1})A_{1}^{\top}\mu_{1,\mathrm{p}}. \nonumber
\end{align}
The right-hand side of (\ref{normalform}) is obtained by evaluating $\nabla_q  g_{\mathrm{d}}(z,q)$ and by using equality conditions \mbox{implied by (\ref{statKKTcond})}. We observe that $n(z,\mu,q)$ is the gradient of $ \Phi(q)$ for configurations where the SDF is differentiable. We show that if $n(z,\mu,q)$ is used to simulate forces or impulses, the ambiguity for nondifferentiable configurations is resolved due to the combined consideration with the Newton-Euler equations.

We formulate a CLS by combining the KKT conditions of (\ref{distprob_compact}) together with the Newton-Euler equations for rigid body dynamics and obtain
\begin{subequations} \label{cls}
	\begin{align}
		&\dot{q} = \nu, \label{cls:a}\\
		&M \dot{\nu} = f_{\mathrm{v}}(q,\nu,u) + n(z,\mu,q)\lambda_{\mathrm{n}}, \label{cls:b}\\
		&0 \leq  f_{\mathrm{d}}(z)  \ \perp \ \lambda_{\mathrm{n}} \geq 0, \label{cls:c} \\
		&\nabla_{z} f_{\mathrm{d}}(z) + \nabla_{z} g_{\mathrm{d}}(z,q) \mu = 0, \label{cls:d} \\
		&0 \leq \mu\ \perp \ -g_{\mathrm{d}}(z,q) \geq 0, \label{cls:e}\\
		&\text{State jump law (Section \ref{section1imp})} \label{cls:f},
	\end{align}
\end{subequations}
where $\nu \in \mathbb{R}^{n_\mathrm{q}}$ is the system velocity and $M \in \mathbb{R}^{n_\mathrm{q} \times n_\mathrm{q}}$ is the inertia matrix, which is positive definite and diagonal. Furthermore, the function $f_{\mathrm{v}}: \mathbb{R}^{n_\mathrm{q}}\times \mathbb{R}^{n_\mathrm{q}} \times \mathbb{R}^{n_\mathrm{u}} \rightarrow \mathbb{R}^{n_\mathrm{q}}$ captures external forces acting on the system, e.g., such as gravity or forces imposed by a control input $u \in \mathbb{R}^{n_\mathrm{u}}$. 

\subsection{Simultaneous Impact Law} \label{section1imp}
In the following, an impact law is derived that determines how the CLS (\ref{cls}) evolves if it enters a state $q(t_{\mathrm{s}})$ with $\Phi(q(t_\mathrm{s})) = 0$ and there exists a solution $(z, \mu)$ to (\ref{distprob}) such that $n(z,\mu,q(t_\mathrm{s}))^{\top} \nu(t_\mathrm{s}^{-})  < 0$. In this case, the system velocities need to undergo an instantaneous jump to avoid penetration of the objects. Here we have to deal with simultaneous impacts, as in the case of a patch contact, it has to be ensured that after an impact all points on the contact patch have a positive relative velocity with respect to the halfspace separating the objects, cf. Fig. \ref{distfig}.

Moreau's impact law is commonly stated for the case where one has $n_{\mathrm{v}} \in \mathbb{N}$ continuously differentiable SDF $\tilde{\Phi}_{i}(q), \ i = 1,\dots, n_{\mathrm{v}}$, that all become active at \mbox{time $t_{\mathrm{s}}$}, i.e., $\tilde{\Phi}_{i}(q(t_\mathrm{s})) = 0$. Given the corresponding contact normals $\tilde{n}_{i}(q) = \nabla_{q}\tilde{\Phi}_{i}(q)$, Moreau's impact law can be stated as a mixed linear complementarity problem (MLCP), \mbox{cf. \cite[Equation 5.65]{brogliato2016}}. In this formulation, one aims to find $\tilde{\Lambda}_{\mathrm{n},i} \in \mathbb{R}, \ i = 1,\dots, n_{\mathrm{v}}$, such that
\begin{subequations} \label{moreau}
\begin{align}
	&\nu(t_\mathrm{s}^{+}) = \nu(t_\mathrm{s}^{-}) + M^{-1} \sum_{i = 1}^{n_{\mathrm{v}}}\tilde{n}_{i}(q(t_\mathrm{s})) \tilde{\Lambda}_{\mathrm{n},i}, \label{moreau:a}\\
	&0 \leq \tilde{\Lambda}_{\mathrm{n},i}, \quad  0 \leq \tilde{n}_{i}(q(t_\mathrm{s}))^{\top}\nu(t_\mathrm{s}^{+}), \quad i = 1,\dots,n_{\mathrm{v}}, \\
	& 0= \left(\sum_{i = 1}^{n_{\mathrm{v}}} \tilde{\Lambda}_{\mathrm{n},i} \tilde{n}_{i}(q(t_\mathrm{s})) \right)^{\top}\nu(t_\mathrm{s}^{+}).
\end{align}
\end{subequations}
Now, for the case of impacts between padded polytopes, one obtains, generally both in two and three dimensions, polytopic contact patches. Therefore, one can locally resolve an impact, by considering the $n_{\mathrm{v}}$ vertices, denoted by $w_{i}$, of the polytopic contact patch in terms of the SDF derived in Section \ref{section1a}. We obtain
\begin{equation*}
\left\{\sum_{i=1}^{n_{\mathrm{v}}}\theta_{i}w_{i}  \mid  \bm{1}^{\top}\theta = 1, \ \theta_{i} \geq 0 \right\} = \{p \mid (z, \mu) \text{ solve }(\ref{distprob}) \},
\end{equation*}
where we remind the reader that $p$ is included in $z$.

The contact patch vertices $w_{i}$ are at the moment of impact elements of both considered bodies. By virtually fixing them to one body and defining the SDF $\tilde{\Phi}_{i}(q)$ such that they encode the point-object distance between $w_{i}$ fixed to one body and the other body, one can apply Moreau's law to solve for the resulting impulse. For general numerical simulation this is impractical as one does not know beforehand which contacts form and has generally no known expression for the vertices $w_{i}$ and the SDF $\tilde{\Phi}_{i}(q)$.

We now consider the SDF defined in Section \ref{section1a} and show that it allows us to encode the above impact law, without the need to locally enumerate vertices of contact patches. To this end, let $(z, \mu)$ solve (\ref{distprob}) at $q(t_\mathrm{s})$, then the corresponding $p$ is an element of the contact patch. We observe that, $n(z,\mu,q)$ is a linear function in $p$ and does not depend on the other components of $z$. Due to the definition of the functions $\tilde{\Phi}_{i}(q)$ it holds for $\tilde{z}_{i} = (w_{i},\dots)$, $n(\tilde{z}_{i},\mu,q(t_\mathrm{s})) = \tilde{n}_{i}(q(t_\mathrm{s}))$. Consequently, for any $(z, \mu)$ solving (\ref{distprob}) at $q(t_\mathrm{s})$ and any $\Lambda_{\mathrm{n}} \geq 0$, there exist \mbox{$\tilde{\Lambda}_{\mathrm{n},i} \geq 0$}, $ \ i = 1,\dots,n_{\mathrm{v}}$, such that
\begin{equation*}
n(z,\mu,q(t_\mathrm{s})) \Lambda_{\mathrm{n}} = \sum_{i = 1}^{n_{\mathrm{v}}}\tilde{n}_{i}(q(t_\mathrm{s})) \tilde{\Lambda}_{\mathrm{n},i}.
\end{equation*}
Now, if it holds
\begin{align*}
	 0 \leq \tilde{n}_{i}(q(t_\mathrm{s}))^{\top}\nu(t_\mathrm{s}^{+}), \quad i = 1,\dots,n_{\mathrm{v}},
\end{align*}
then for all $\hat{\Lambda}_{\mathrm{n},i}\geq 0, \ i = 1,\dots, n_{\mathrm{v}}$, it follows
\begin{align*}
	 0 \leq \left(\sum_{i = 1}^{n_{\mathrm{v}}} \hat{\Lambda}_{\mathrm{n},i} \tilde{n}_{i}(q(t_\mathrm{s})) \right)^{\top}\nu(t_\mathrm{s}^{+}),
\end{align*}
and vice versa. Consequently, solving (\ref{moreau}) is equivalent to determining $\Lambda_{\mathrm{n}} \geq 0$ and $(z_{\mathrm{I}},\mu_{\mathrm{I}})$ solving (\ref{distprob}) at $q(t_\mathrm{s})$ such that
\begin{subequations}  \label{impactres}
	\begin{align}
		&\nu(t_{\mathrm{s}}^{+}) = \nu(t_{\mathrm{s}}^{-}) + M^{-1} n(z_{\mathrm{I}},\mu_{\mathrm{I}},q(t_{\mathrm{s}})) \Lambda_{\mathrm{n}}, \label{impactres:a} \\
		&0 \leq n(\hat{z}_{\mathrm{I}},\hat{\mu}_{\mathrm{I}},q(t_{\mathrm{s}}))^{\top} v(t_{\mathrm{s}}^{+}),  \ \forall \, (\hat{z}_{\mathrm{I}},\hat{\mu}_{\mathrm{I}}) \text{ solving } (\ref{distprob}), \label{impactres:b}\\
		& 0 = n(z_{\mathrm{I}},\mu_{\mathrm{I}},q(t_{\mathrm{s}}))^{\top} v(t_{\mathrm{s}}^{+}). \label{impactres:c}
	\end{align}
\end{subequations}
Since Moreau's impact law provides unique post-impact velocities $\nu(t_{\mathrm{s}}^{+})$ \cite[Corollary 5.1]{brogliato2016}, (\ref{impactres:a}) implies that $z_{\mathrm{I}}$ and in particular $p_{\mathrm{I}}$ are also unique. Thus, $p_{\mathrm{I}}$ can be viewed as a unique ECP for impact resolution.

Now, numerical evaluation of (\ref{impactres:b}) is generally not possible as we do not have a representation of the solution set of (\ref{distprob}), but only have access to single evaluations. Since this condition encodes that all points on the contact surface have a post-impact velocity such that no penetration occurs, we propose to substitute (\ref{impactres:b}) by the condition
\begin{equation} \label{expeulerstep}
	0 \leq \Phi(q(t_{\mathrm{s}}) + h_{\mathrm{E}} \nu(t_{\mathrm{s}}^{+})),
\end{equation}
where $0 < h_{\mathrm{E}}  << 1$ is a small positive constant. Thus, we consider one additional evaluation of the SDF by using an explicit Euler step over a small time interval. If the objects do not intersect an arbitrarily small amount of time after the impulse has acted, (\ref{impactres:b}) was fulfilled. 

\subsection{Well-Posedness of CLS and Uniqueness of ECP} \label{section1unique}
In the following, we show that the CLS (\ref{cls}) is well-posed and admits unique state trajectories $(q,\nu)$. This is not directly clear as the here considered SDF may yield many possible candidates for the ECP $p$, which in turn could result in ambiguous contact forces that are applied to the system.

If we have $f_{\mathrm{d}}(z) > 0$, then then it follows $\lamn = 0$. Therefore, the primal-dual variables of the embedded distance problem are decoupled from the dynamics and there may be multiple primal variables solving the system.

Now assume $\lambda_{\mathrm{n}} > 0$ for some time interval $(t_1,t_2)$, then we have $f_{\mathrm{d}}(z) = 0$ and contact forces are active to prevent penetration of the objects. The ECP $p$ is by definition an element of the contact patch and $\ntr$ induces a separating hyperplane between the objects. If we virtually fix $p$ onto either object, its velocity in world coordinates is given by
\begin{align*}
	V_{1} &= v_{1} + \omega_{1} \times (p - c_{1}), \\
	V_{2} &= v_{2} + \omega_{2} \times (p - c_{2}).
\end{align*}
I.e., $V_{1}$ and $V_{2}$ describes how $p$ would move along with the first or second object, respectively. If contact persists during $(t_1,t_2)$, we conclude that the condition 
\begin{equation} \label{contactzerovel}
	n(z, \mu, q)^{\top}\nu = n_{\mathrm{tr}}^{\top} (V_{1}(p) - V_{2}(p)) = 0
\end{equation}
has to hold, as the ECP $p$ can only have relative velocity along the separating halfspace between the objects. 

The following proposition is the main result that ensures uniqueness of state trajectories.
\begin{proposition} \label{propunique}
	On every time interval $(t_{1},t_{2})$ during which the CLS (\ref{cls}) evolves with $\lamn > 0$, no impact occurs and the active set of the SDF does not change, it has a unique solution.
\end{proposition}
A full proof can be found in Appendix \ref{app:proof2}. We briefly outline the key steps of the proof in the following.

Under the assumptions of Proposition \ref{propunique}, condition (\ref{contactzerovel}) applies. Thus, the CLS (\ref{cls}) evolves according to the DAE
\begin{subequations} \label{contactDAE}
	\begin{align} 
		&\dot{q} = \nu, \label{contactDAE:a} \\
		&M \dot{\nu} = f_{\mathrm{v}}(q,\nu,u) + n(z, \mu, q)\lambda_{\mathrm{n}}, \label{contactDAE:b} \\
		&f_{\mathrm{d}}(z) = 0, \\
		&n(z, \mu, q)^{\top}\nu = 0, \label{contactDAE:d}\\
		&\nabla_{z} f_{\mathrm{d}}(z) + \nabla_{z} g_{\mathrm{d}}(z,q) \mu = 0, \label{contactDAE:e}\\
		&0 = \mu_{\mathcal{I}(q)}, \label{contactDAE:f}\\
		&0 = g_{\mathrm{d}}(z,q)_{\mathcal{A}(q)}, \label{contactDAE:g}
	\end{align}
\end{subequations}
where $\mathcal{A}(q) \subset \{1,\dots,\nha+\nhb + 2\}$ and $\mathcal{I}(q) = \{1,\dots,\nha+\nhb + 2\} \setminus \mathcal{A}(q)$ denote the index sets for the active and inactive constraints of (\ref{distprob}).

For those active sets of the SDF without patch contacts, primal and dual variables of the distance problem are unique. The KKT system (\ref{contactDAE:e})-(\ref{contactDAE:g}) then is a smooth system of equations with an invertible Jacobian. The implicit function theorem yields continuously differentiable functions for $z$ and $\mu$ in terms of the differential states and the control input. Differentiating (\ref{contactDAE:d}) with respect to time yields a similar differentiable function for $\lamn$, c.f., \cite{tflag}. Therefore, the DAE (\ref{contactDAE}) can be equivalently stated as an ODE with a differentiable right-hand side, and the Picard-Lindelöf theorem implies uniqueness of solutions.

The more challenging case is the one, where the sets $\mathcal{A}(q)$ and $\mathcal{I}(q)$ are such that patch contact is maintained. Then the distance problem has multiple viable primal solutions. The stationary KKT condition (\ref{contactDAE:e}) implies in this case
\begin{equation} \label{uniquecondrot}
	R(\xi_{1})A^{\top}_{1,k_{1}} + R(\xi_{2})A^{\top}_{2,k_{2}}  = 0,
\end{equation}
where $A^{\top}_{i,j}$ is the $j$-th row of the matrix from the halfspace representation of the $i$-th object and $\{k_{1}\}$,$\{k_{2}\}$, are the singleton active sets of the constraints (\ref{distprob:b}), (\ref{distprob:c}), respectively. By differentiating (\ref{uniquecondrot}) twice with respect to time and by using equality conditions given by (\ref{contactDAE:e}), one obtains
\begin{equation} \label{rotacceq}
\dot{\omega}_{1} = \dot{\omega}_{2}.
\end{equation}
Equation (\ref{rotacceq}) then implies that the force components in (\ref{contactDAE:b}) acting on the rotational DoF have to be equal. Using this, and again differentiating (\ref{contactDAE:d}) with respect to time yields differentiable expressions for $z$, $\mu$ and $\lamn$ in terms of the differential states and the control input, which again allows use of the Picard-Lindelöf theorem.

We note that since the system velocity is obtained from integration of the right-hand side of (\ref{contactDAE:b}), switches of the SDF active set do not hinder uniqueness of state trajectories. Overall, one can conclude that state trajectories for the CLS (\ref{cls}) evolve uniquely, as the considered impact law provides unique post-impact velocities and for every fixed set of the complementarities (\ref{cls:c}) and (\ref{cls:e}), unique evolution is ensured by the Picard-Lindelöf theorem. In particular, ECP which parameterize impulse and force interactions are well-defined.
\section{FINITE ELEMENTS WITH SWITCH DETECTION AND STATE JUMPS} \label{secfesd}
In this section, we adapt the FESD-J \cite{fesdj} method to the considered CLS (\ref{cls}). To this end, in Section \ref{sec:rki} a standard RK discretization is stated. The discretized impact law is discussed in Section \ref{sec:discimp}. In Section \ref{sec:crosscomp}, so-called cross complementarity conditions are formulated, which ensure that dynamic switches are detected exactly in time. Finally, in Section \ref{sec:discopt} the derived discretization scheme is integrated in a discrete time optimal control problem (OCP).

\subsection{Runge-Kutta Integration} \label{sec:rki}
We consider a control interval $[0,T]$, with a constant control input $\hat{u} \in \mathbb{R}^{n_{\mathrm{u}}}$, that is divided into $N_{\mathrm{FE}}$ \textit{finite elements} (FE) $[t_{n}, t_{n+1}]$, with $t_{n} < t_{n+1}$ and $\cup_{n = 0}^{N_{\mathrm{FE}}} [t_{n}, t_{n+1}] = [0,T]$. On each FE, we consider a $n_{\mathrm{s}}$-stage RK method that is parameterized by its Butcher tableau entries $\hat{a}_{i,j}, \hat{b}_{i}, \hat{c}_{i}$, with $i,j \in \{1,\dots,n_{\mathrm{s}} \}$ \cite{hairer}. We denote the FE lengths by $h_{n} = t_{n+1} - t_{n}$. Approximations of the differential and algebraic states for the CLS (\ref{cls}) at time $t_{n,i} = t_{n} + \hat{c}_{i} h_{n}$ are denoted by $q_{n,i}$, $\nu_{n,i}$, $\lambda_{\mathrm{n},n,i}$, $z_{n,i}$ and $\mu_{n,i}$, respectively. For the differential states, we additionally introduce the values $q_{n,0}$, $\nu_{n,0}$, which approximate the states at $q(t_{\mathrm{n}}^{+})$, $\nu(t_{\mathrm{n}}^{+})$, while $q_{n-1,n_{\mathrm{s}}}$, $\nu_{n-1,n_{\mathrm{s}}}$, approximate $q(t_{\mathrm{n}}^{-})$, $\nu(t_{\mathrm{n}}^{-})$. For ease of exposition, we assumed that it holds $\hat{c}_{n_{\mathrm{s}}} = 1$, c.f., \cite{fesdj} for the case $\hat{c}_{n_{\mathrm{s}}} < 1$. An example for RK methods that fulfill $\hat{c}_{n_{\mathrm{s}}} = 1$ is given by Radau IIA schemes. Introducing additional variables for the differential states is required as due to possible state jumps in velocity for the CLS (\ref{cls}), we may have $\nu(t_{\mathrm{s}}^{-}) \neq \nu(t_{\mathrm{s}}^{+})$ for some $t_{\mathrm{s}}$. Now let
\begin{equation*}
F_{\mathrm{v}}(q,\nu,z,\mu,\lambda_{\mathrm{n}},u) = M^{-1}( f_{\mathrm{v}}(q,\nu,u) + n(z,\mu,q)\lambda_{\mathrm{n}}),
\end{equation*}
then the discretized differential equations of (\ref{cls:a})-(\ref{cls:b}) read
\begin{equation} \label{rkdiff}
\begin{split}
 q_{n,i} &= q_{n,0} + h_{n} \sum _{j = 1}^{n_{\mathrm{s}}} \hat{a}_{i,j} \nu_{n,j}, \\
 \nu_{n,i} &= \nu_{n,0} + h_{n} \sum _{j = 1}^{n_{\mathrm{s}}} \hat{a}_{i,j} F_{\mathrm{v}}(q_{n,j},\nu_{n,j},z_{n,j},\mu_{n,j},\lambda_{\mathrm{n},n,j},\hat{u}),
\end{split}
\end{equation}
where an initial state $s_{0} = (q_{0,0},\nu_{0,0})$ at time $t_{0}$ has to be provided. The algebraic conditions in (\ref{cls}) are also evaluated at the RK stage points and read
\begin{subequations} \label{rkalg}
\begin{align}
	&0 \leq  f_{\mathrm{d}}(z_{n,i})  \perp \lambda_{\mathrm{n},n,i} \geq 0,  \label{rkalg:a} \\
	&\nabla_{z} f_{\mathrm{d}}(z_{n,i}) + \nabla_{z} g_{\mathrm{d}}(z_{n,i},q_{n,i}) \mu_{n,i} = 0, \\
	&0 \leq \mu_{n,i} \perp -g_{\mathrm{d}}(z_{n,i},q_{n,i}) \geq 0. \label{rkalg:c}
\end{align}
\end{subequations}

\subsection{Impact Equations} \label{sec:discimp}
We observe that the configurations are continuous in time as they are obtained by integration of the velocities (\ref{cls:a}). Therefore, we impose
\begin{equation*}
q_{n-1,n_{\mathrm{s}}} = q_{n,0}.
\end{equation*}
To connect the velocities at FE boundaries, the impact law (\ref{impactres})-(\ref{expeulerstep}) is incorporated into the discretization. To this end, two evaluations of the SDF at FE boundaries are required. This is achieved by evaluating the KKT conditions of (\ref{distprob}),
\begin{equation} \label{impKKT}
	\begin{split}
		&\nabla_{z} f_{\mathrm{d}}(z_{\mathrm{I},n}) + \nabla_{z} g_{\mathrm{d}}(z_{\mathrm{I},n}, q_{n,0}) \mu_{\mathrm{I},n} = 0, \\
		&0 \leq \mu_{\mathrm{I},n} \perp -g_{\mathrm{d}}(z_{\mathrm{I},n},q_{n,0}) \geq 0, \\
		&\nabla_{z} f_{\mathrm{d}}(z_{\mathrm{E},n}) + \nabla_{z} g_{\mathrm{d}}(z_{\mathrm{E},n}, q_{n,0} + h_{\mathrm{E}}\nu_{n,0} ) \mu_{\mathrm{E},n} = 0, \\
		&0 \leq \mu_{\mathrm{E},n} \perp -g_{\mathrm{d}}(z_{\mathrm{E},n},q_{n,0} + h_{\mathrm{E}}\nu_{n,0}) \geq 0,
	\end{split}
\end{equation}
with a small constant $h_{\mathrm{E}} > 0$. Thus, $(z_{\mathrm{I},n},\mu_{\mathrm{I},n})$ solves (\ref{distprob}) for $q_{n,0}$ and $(z_{\mathrm{E},n},\mu_{\mathrm{E},n})$ solves (\ref{distprob}) for $q_{n,0} + h_{\mathrm{E}}\nu_{n,0}$. 
The impact equations are given by
\begin{subequations} \label{discimpeq}
\begin{align}
\nu_{n,0} &= \nu_{n-1,n_{\mathrm{s}}} + M^{-1}n(z_{\mathrm{I},n},\mu_{\mathrm{I},n},q_{n,0}) \Lambda_{\mathrm{n},n}, \label{discimpeq:a}\\
0 &= \Lambda_{\mathrm{n},n} n(z_{\mathrm{I},n},\mu_{\mathrm{I},n},q_{n,0})^{\top} \nu_{n,0}, \\
0 &\leq f_{\mathrm{d}}(z_{\mathrm{E},n}), \label{discimpeq:c} \\
0 &\leq \Lambda_{\mathrm{n},n} \perp f_{\mathrm{d}}(z_{\mathrm{I},n}) \geq 0. \label{discimpeq:d} 
\end{align}
\end{subequations}
If $f_{\mathrm{d}}(z_{\mathrm{I},n}) = 0$ and there exist $(z_{\mathrm{I},n},\mu_{\mathrm{I},n})$ with $n(z_{\mathrm{I},n},\mu_{\mathrm{I},n},q_{n,0})^{\top} \nu_{n-1,n_{\mathrm{s}}} < 0$, then it holds $\Lambda_{\mathrm{n},n} > 0 $ as otherwise (\ref{discimpeq:c}) will be violated. In this case, (\ref{discimpeq:a})-(\ref{discimpeq:c}) encode precisely the impact law discussed in Section \ref{section1imp}. Otherwise, we have $\Lambda_{\mathrm{n},n} = 0$ and (\ref{discimpeq:a}) encodes continuity of the velocities across the FE boundary.

We collect all discrete time differential variables $q_{n,i}$, $\nu_{n,i}$, in the vector $\bm{\mathrm{x}}$, discrete time algebraic variables $z_{n,j},\mu_{n,j},\lambda_{\mathrm{n},n,j},z_{\mathrm{I},n},\mu_{\mathrm{I},n},z_{\mathrm{E},n},\mu_{\mathrm{E},n}$, in the vector $\bm{\mathrm{z}}$ and FE lengths in the vector $\bm{\mathrm{h}}$ such that the conditions (\ref{rkdiff}), (\ref{rkalg}), (\ref{impKKT}), (\ref{discimpeq}), are summarized by $G_{\mathrm{rk}}(\bm{\mathrm{x}},\bm{\mathrm{z}},\bm{\mathrm{h}},s_{0},\hat{u}) = 0.$

Note that complementarity conditions can be equivalently denoted by so-called C-functions $\Psi(\cdot,\cdot)$, which have the property: $0 \leq a \perp b \geq 0 \Leftrightarrow \Psi(a,b) = 0$.
\begin{figure*}
	\captionsetup[subfigure]{justification=centering}
	\centering
	\begin{subfigure}[t]{0.18\textwidth}
		\centering
\begin{tikzpicture}

\tikzstyle{every node}=[font=\scriptsize]

\definecolor{darkgray176}{RGB}{176,176,176}
\definecolor{gray}{RGB}{150,150,150}

\begin{axis}[
width=1.4\textwidth,
height=1.01\textwidth,
tick align=outside,
tick pos=left,
x grid style={darkgray176},
xlabel= {\(\displaystyle c_{\mathrm{x}}\)},
xmin=-1.5, xmax=3.5,
xtick = {-1,0,1,2,3},
xtick style={color=black},
y grid style={darkgray176},
ylabel= {\(\displaystyle c_{\mathrm{y}}\)},
ymin=1, ymax=4,
ytick style={color=black},
x label style={at={(axis description cs:0.5,-0.27)},anchor=north},
y label style={at={(axis description cs:-0.12,.5)}, anchor=south},
]
\path [draw=black, fill=mediumcyan, opacity=0.2]
(axis cs:-1.13767236989995,3.11852237946576)
--(axis cs:0.789613224465101,3.50920251073486)
.. controls (axis cs:0.920288768177266,3.53569175484086) and (axis cs:0.998871162086353,3.48359860503778) .. (axis cs:1.02536040619236,3.35292306132562)
--(axis cs:1.13395181930918,2.8172248022615)
.. controls (axis cs:1.16044106341519,2.68654925854933) and (axis cs:1.10834791361211,2.60796686464025) .. (axis cs:0.977672369899947,2.58147762053424)
--(axis cs:-0.949613224465101,2.19079748926515)
.. controls (axis cs:-1.08028876817727,2.16430824515914) and (axis cs:-1.15887116208635,2.21640139496222) .. (axis cs:-1.18536040619236,2.34707693867438)
--(axis cs:-1.29395181930918,2.8827751977385)
.. controls (axis cs:-1.32044106341519,3.01345074145067) and (axis cs:-1.26834791361211,3.09203313535975) .. (axis cs:-1.13767236989995,3.11852237946576);
\path [draw=black, fill=gray, opacity=0.2]
(axis cs:-2.16648435722623,2.18324217861312)
--(axis cs:2.16648435722623,2.18324217861312)
.. controls (axis cs:2.29981769055957,2.18324217861312) and (axis cs:2.36648435722623,2.11657551194645) .. (axis cs:2.36648435722623,1.98324217861312)
--(axis cs:2.36648435722623,0.0167578213868838)
.. controls (axis cs:2.36648435722623,-0.11657551194645) and (axis cs:2.29981769055957,-0.183242178613116) .. (axis cs:2.16648435722623,-0.183242178613116)
--(axis cs:-2.16648435722623,-0.183242178613116)
.. controls (axis cs:-2.29981769055957,-0.183242178613116) and (axis cs:-2.36648435722623,-0.11657551194645) .. (axis cs:-2.36648435722623,0.0167578213868838)
--(axis cs:-2.36648435722623,1.98324217861312)
.. controls (axis cs:-2.36648435722623,2.11657551194645) and (axis cs:-2.29981769055957,2.18324217861312) .. (axis cs:-2.16648435722623,2.18324217861312);
\path [draw=black, fill=mediumcyan]
(axis cs:-0.943520994591018,3.08309116650045)
--(axis cs:0.624585529954969,3.40096209577255)
.. controls (axis cs:0.755261073667134,3.42745133987855) and (axis cs:0.833843467576221,3.37535819007548) .. (axis cs:0.860332711682229,3.24468264636331)
--(axis cs:0.939800444000254,2.85265601522681)
.. controls (axis cs:0.966289688106262,2.72198047151465) and (axis cs:0.914196538303183,2.64339807760556) .. (axis cs:0.783520994591018,2.61690883349955)
--(axis cs:-0.784585529954969,2.29903790422745)
.. controls (axis cs:-0.915261073667134,2.27254866012145) and (axis cs:-0.993843467576221,2.32464180992452) .. (axis cs:-1.02033271168223,2.45531735363669)
--(axis cs:-1.09980044400025,2.84734398477319)
.. controls (axis cs:-1.12628968810626,2.97801952848535) and (axis cs:-1.07419653830318,3.05660192239444) .. (axis cs:-0.943520994591018,3.08309116650045);
\path [draw=black, fill=gray]
(axis cs:-1.8,2)
--(axis cs:1.8,2)
.. controls (axis cs:1.93333333333333,2) and (axis cs:2,1.93333333333333) .. (axis cs:2,1.8)
--(axis cs:2,0.2)
.. controls (axis cs:2,0.0666666666666667) and (axis cs:1.93333333333333,0) .. (axis cs:1.8,0)
--(axis cs:-1.8,0)
.. controls (axis cs:-1.93333333333333,0) and (axis cs:-2,0.0666666666666667) .. (axis cs:-2,0.2)
--(axis cs:-2,1.8)
.. controls (axis cs:-2,1.93333333333333) and (axis cs:-1.93333333333333,2) .. (axis cs:-1.8,2);

\addplot [semithick, mediumblue, opacity=0.2, mark=*, mark size=2.5, mark options={solid}, only marks]
table {%
-0.982260048211744 2.18324217659434
};
\addplot [semithick, mediumorange, opacity=0.2, mark=*, mark size=1.3, mark options={solid}, only marks]
table {%
-0.982260048211744 2.18324217659434
};
\end{axis}

\end{tikzpicture}
		\vspace{-1.5em}
		\caption{$t=0$}
		\label{simvis:a}
	\end{subfigure}
	\hspace{0.6em}
	\hfill
	\begin{subfigure}[t]{0.18\textwidth}
		\centering
\begin{tikzpicture}

\tikzstyle{every node}=[font=\scriptsize]

\definecolor{darkgray176}{RGB}{176,176,176}
\definecolor{gray}{RGB}{150,150,150}

\begin{axis}[
width=1.4\textwidth,
height=1.01\textwidth,
tick align=outside,
tick pos=left,
x grid style={darkgray176},
xlabel= {\(\displaystyle c_{\mathrm{x}}\)},
xmin=-1.5, xmax=3.5,
xtick = {-1,0,1,2,3},
xtick style={color=black},
y grid style={darkgray176},
yticklabels={,,},
ymin=1, ymax=4,
ytick style={color=black},
x label style={at={(axis description cs:0.5,-0.27)},anchor=north},
y label style={at={(axis description cs:-0.12,.5)}, anchor=south},
]
\path [draw=black, fill=mediumcyan]
(axis cs:-0.526950006204217,2.78803998100841)
--(axis cs:1.04115651894029,3.1059109073279)
.. controls (axis cs:1.17183206270234,3.13240015118786) and (axis cs:1.25041445651334,3.08030700123682) .. (axis cs:1.27690370037329,2.94963145747478)
--(axis cs:1.35637143195317,2.55760482618865)
.. controls (axis cs:1.38286067581312,2.42692928242661) and (axis cs:1.33076752586208,2.34834688861561) .. (axis cs:1.20009198210004,2.32185764475565)
--(axis cs:-0.368014543044472,2.00398671843616)
.. controls (axis cs:-0.498690086806514,1.9774974745762) and (axis cs:-0.577272480617514,2.02959062452724) .. (axis cs:-0.603761724477472,2.16026616828929)
--(axis cs:-0.683229456057345,2.55229279957541)
.. controls (axis cs:-0.709718699917302,2.68296834333746) and (axis cs:-0.657625549966259,2.76155073714846) .. (axis cs:-0.526950006204217,2.78803998100841);
\path [draw=black, fill=gray]
(axis cs:-1.8,2)
--(axis cs:1.8,2)
.. controls (axis cs:1.93333333333333,2) and (axis cs:2,1.93333333333333) .. (axis cs:2,1.8)
--(axis cs:2,0.2)
.. controls (axis cs:2,0.0666666666666667) and (axis cs:1.93333333333333,0) .. (axis cs:1.8,0)
--(axis cs:-1.8,0)
.. controls (axis cs:-1.93333333333333,0) and (axis cs:-2,0.0666666666666667) .. (axis cs:-2,0.2)
--(axis cs:-2,1.8)
.. controls (axis cs:-2,1.93333333333333) and (axis cs:-1.93333333333333,2) .. (axis cs:-1.8,2);
\addplot [semithick, mediumred, mark=diamond*, mark size=4, mark options={solid}, only marks]
table {%
-0.407748427403386 2.00000002148254
};
\addplot [semithick, mediumblue, mark=*, mark size=2.5, mark options={solid}, only marks]
table {%
-0.403573220542818 2.00000000185421
};
\addplot [semithick, mediumorange, mark=*, mark size=1.3, mark options={solid}, only marks]
table {%
-0.403573220542818 2.00000000185421
};
\end{axis}

\end{tikzpicture}
		\vspace{-1.5em}
		\caption{$t=0.42$}
		\label{simvis:b}
	\end{subfigure}
	\hfill
	\begin{subfigure}[t]{0.18\textwidth}
		\centering
\begin{tikzpicture}

\tikzstyle{every node}=[font=\scriptsize]

\definecolor{darkgray176}{RGB}{176,176,176}
\definecolor{gray}{RGB}{150,150,150}

\begin{axis}[
width=1.4\textwidth,
height=1.01\textwidth,
tick align=outside,
tick pos=left,
x grid style={darkgray176},
xlabel= {\(\displaystyle c_{\mathrm{x}}\)},
xmin=-1.5, xmax=3.5,
xtick = {-1,0,1,2,3},
xtick style={color=black},
y grid style={darkgray176},
yticklabels={,,},
ymin=1, ymax=4,
ytick style={color=black},
x label style={at={(axis description cs:0.5,-0.27)},anchor=north},
y label style={at={(axis description cs:-0.12,.5)}, anchor=south},
]
\path [draw=black, fill=mediumcyan]
(axis cs:-0.214985485971947,2.79999999133572)
--(axis cs:1.38501451402805,2.7999999856438)
.. controls (axis cs:1.51834784736139,2.79999998516947) and (axis cs:1.58501451379089,2.73333331826564) .. (axis cs:1.58501451331656,2.59999998493231)
--(axis cs:1.58501451189358,2.19999998493231)
.. controls (axis cs:1.58501451141926,2.06666665159898) and (axis cs:1.51834784451543,1.99999998516947) .. (axis cs:1.38501451118209,1.9999999856438)
--(axis cs:-0.214985488817906,1.99999999133572)
.. controls (axis cs:-0.348318822151239,1.99999999181005) and (axis cs:-0.414985488580743,2.06666665871388) .. (axis cs:-0.414985488106416,2.19999999204721)
--(axis cs:-0.414985486683437,2.59999999204721)
.. controls (axis cs:-0.41498548620911,2.73333332538054) and (axis cs:-0.348318819305281,2.79999999181005) .. (axis cs:-0.214985485971947,2.79999999133572);
\path [draw=black, fill=gray]
(axis cs:-1.8,2)
--(axis cs:1.8,2)
.. controls (axis cs:1.93333333333333,2) and (axis cs:2,1.93333333333333) .. (axis cs:2,1.8)
--(axis cs:2,0.2)
.. controls (axis cs:2,0.0666666666666667) and (axis cs:1.93333333333333,0) .. (axis cs:1.8,0)
--(axis cs:-1.8,0)
.. controls (axis cs:-1.93333333333333,0) and (axis cs:-2,0.0666666666666667) .. (axis cs:-2,0.2)
--(axis cs:-2,1.8)
.. controls (axis cs:-2,1.93333333333333) and (axis cs:-1.93333333333333,2) .. (axis cs:-1.8,2);
\addplot [semithick, mediumred, mark=diamond*, mark size=4, mark options={solid}, only marks]
table {%
1.06834805332321 1.99999999338486
};
\addplot [semithick, mediumblue, mark=*, mark size=2.5, mark options={solid}, only marks]
table {%
0.596917719382176 1.9999999993419
};

\addplot [semithick, mediumorange, mark=*, mark size=1.3, mark options={solid}, only marks]
table {%
-0.214985465670403 1.99999999449103
};

\end{axis}

\end{tikzpicture}
		\vspace{-1.5em}
		\caption{$t=0.67$}
		\label{simvis:c}
	\end{subfigure}
	\hfill
	\begin{subfigure}[t]{0.18\textwidth}
		\centering
\begin{tikzpicture}

\tikzstyle{every node}=[font=\scriptsize]

\definecolor{darkgray176}{RGB}{176,176,176}
\definecolor{gray}{RGB}{150,150,150}

\begin{axis}[
width=1.4\textwidth,
height=1.01\textwidth,
tick align=outside,
tick pos=left,
x grid style={darkgray176},
xlabel= {\(\displaystyle c_{\mathrm{x}}\)},
xmin=-1.5, xmax=3.5,
xtick = {-1,0,1,2,3},
xtick style={color=black},
y grid style={darkgray176},
yticklabels={,,},
ymin=1, ymax=4,
ytick style={color=black},
x label style={at={(axis description cs:0.5,-0.27)},anchor=north},
y label style={at={(axis description cs:-0.12,.5)}, anchor=south},
]
\path [draw=black, fill=mediumcyan]
(axis cs:0.99827775333835,2.79999998664617)
--(axis cs:2.59827775333835,2.80000001443445)
.. controls (axis cs:2.73161108667168,2.80000001675014) and (axis cs:2.79827775449619,2.73333335124132) .. (axis cs:2.79827775681189,2.60000001790799)
--(axis cs:2.79827776375896,2.20000001790799)
.. controls (axis cs:2.79827776607465,2.06666668457466) and (axis cs:2.73161110056583,2.00000001675014) .. (axis cs:2.59827776723249,2.00000001443445)
--(axis cs:0.998277767232494,1.99999998664617)
.. controls (axis cs:0.864944433899161,1.99999998433047) and (axis cs:0.798277766074649,2.0666666498393) .. (axis cs:0.798277763758958,2.19999998317263)
--(axis cs:0.798277756811886,2.59999998317263)
.. controls (axis cs:0.798277754496195,2.73333331650596) and (axis cs:0.864944420005016,2.79999998433047) .. (axis cs:0.99827775333835,2.79999998664617);
\path [draw=black, fill=gray]
(axis cs:-1.8,2)
--(axis cs:1.8,2)
.. controls (axis cs:1.93333333333333,2) and (axis cs:2,1.93333333333333) .. (axis cs:2,1.8)
--(axis cs:2,0.2)
.. controls (axis cs:2,0.0666666666666667) and (axis cs:1.93333333333333,0) .. (axis cs:1.8,0)
--(axis cs:-1.8,0)
.. controls (axis cs:-1.93333333333333,0) and (axis cs:-2,0.0666666666666667) .. (axis cs:-2,0.2)
--(axis cs:-2,1.8)
.. controls (axis cs:-2,1.93333333333333) and (axis cs:-1.93333333333333,2) .. (axis cs:-1.8,2);
\addplot [semithick, mediumblue, mark=*, mark size=2.2, mark options={solid}, only marks]
table {%
1.79942311767792 2.00000001045775
};
\addplot [semithick, mediumorange, mark=*, mark size=1.3, mark options={solid}, only marks]
table {%
1.79942311767792 2.00000001045775
};
\end{axis}

\end{tikzpicture}
		\vspace{-1.5em}
		\caption{$t=1.88$}
		\label{simvis:d}
	\end{subfigure}
	\hfill
	\begin{subfigure}[t]{0.18\textwidth}
		\centering
\begin{tikzpicture}

\tikzstyle{every node}=[font=\scriptsize]

\definecolor{darkgray176}{RGB}{176,176,176}
\definecolor{gray}{RGB}{150,150,150}

\begin{axis}[
width=1.4\textwidth,
height=1.01\textwidth,
tick align=outside,
tick pos=left,
x grid style={darkgray176},
xlabel= {\(\displaystyle c_{\mathrm{x}}\)},
xmin=-1.5, xmax=3.5,
xtick = {-1,0,1,2,3},
xtick style={color=black},
y grid style={darkgray176},
yticklabels={,,},
ymin=1, ymax=4,
ytick style={color=black},
x label style={at={(axis description cs:0.5,-0.27)},anchor=north},
y label style={at={(axis description cs:-0.12,.5)}, anchor=south},
]
\path [draw=black, fill=mediumcyan, opacity=0.2]
(axis cs:1.94144400910807,2.78996792408914)
--(axis cs:3.52545242491969,2.56356340738226)
.. controls (axis cs:3.65744431848148,2.54469762322809) and (axis cs:3.71400737318529,2.46926878437012) .. (axis cs:3.69514158903113,2.33727689080833)
--(axis cs:3.63853819382587,1.9412589328949)
.. controls (axis cs:3.61967240967171,1.80926703933311) and (axis cs:3.54424357081373,1.7527039846293) .. (axis cs:3.41225167725195,1.77156976878346)
--(axis cs:1.82824326144033,1.99797428549034)
.. controls (axis cs:1.69625136787854,2.01684006964451) and (axis cs:1.63968831317473,2.09226890850248) .. (axis cs:1.65855409732889,2.22426080206427)
--(axis cs:1.71515749253415,2.6202787599777)
.. controls (axis cs:1.73402327668831,2.75227065353949) and (axis cs:1.80945211554629,2.8088337082433) .. (axis cs:1.94144400910807,2.78996792408914);
\path [draw=black, fill=gray, opacity=0.2]
(axis cs:-1.80010676723377,2.00005338361689)
--(axis cs:1.80010676723377,2.00005338361689)
.. controls (axis cs:1.93344010056711,2.00005338361689) and (axis cs:2.00010676723377,1.93338671695022) .. (axis cs:2.00010676723377,1.80005338361689)
--(axis cs:2.00010676723377,0.199946616383114)
.. controls (axis cs:2.00010676723377,0.0666132830497803) and (axis cs:1.93344010056711,-5.33836168863377e-05) .. (axis cs:1.80010676723377,-5.33836168863377e-05)
--(axis cs:-1.80010676723377,-5.33836168863377e-05)
.. controls (axis cs:-1.93344010056711,-5.33836168863377e-05) and (axis cs:-2.00010676723377,0.0666132830497803) .. (axis cs:-2.00010676723377,0.199946616383114)
--(axis cs:-2.00010676723377,1.80005338361689)
.. controls (axis cs:-2.00010676723377,1.93338671695022) and (axis cs:-1.93344010056711,2.00005338361689) .. (axis cs:-1.80010676723377,2.00005338361689);
\path [draw=black, fill=mediumcyan]
(axis cs:1.94149383427178,2.78993923204665)
--(axis cs:3.52539655701323,2.56354982219667)
.. controls (axis cs:3.65738845057502,2.54468403804251) and (axis cs:3.71395150527883,2.46925519918453) .. (axis cs:3.69508572112466,2.33726330562275)
--(axis cs:3.63848836866217,1.94128762493738)
.. controls (axis cs:3.619622584508,1.8092957313756) and (axis cs:3.54419374565003,1.75273267667178) .. (axis cs:3.41220185208824,1.77159846082595)
--(axis cs:1.82829912934679,1.99798787067593)
.. controls (axis cs:1.696307235785,2.01685365483009) and (axis cs:1.63974418108119,2.09228249368807) .. (axis cs:1.65860996523536,2.22427438724985)
--(axis cs:1.71520731769785,2.62025006793522)
.. controls (axis cs:1.73407310185202,2.752241961497) and (axis cs:1.80950194070999,2.80880501620081) .. (axis cs:1.94149383427178,2.78993923204665);
\path [draw=black, fill=gray]
(axis cs:-1.8,2)
--(axis cs:1.8,2)
.. controls (axis cs:1.93333333333333,2) and (axis cs:2,1.93333333333333) .. (axis cs:2,1.8)
--(axis cs:2,0.2)
.. controls (axis cs:2,0.0666666666666667) and (axis cs:1.93333333333333,0) .. (axis cs:1.8,0)
--(axis cs:-1.8,0)
.. controls (axis cs:-1.93333333333333,0) and (axis cs:-2,0.0666666666666667) .. (axis cs:-2,0.2)
--(axis cs:-2,1.8)
.. controls (axis cs:-2,1.93333333333333) and (axis cs:-1.93333333333333,2) .. (axis cs:-1.8,2);
\addplot [semithick, mediumblue, mark=*, mark size=2.2, mark options={solid}, only marks]
table {%
1.83379657644112 1.99719545225418
};
\addplot [semithick, mediumorange, mark=*, mark size=1.3, mark options={solid}, only marks]
table {%
1.83379657644112 1.99719545225418
};
\end{axis}

\end{tikzpicture}
		\vspace{-1.5em}
		\caption{$t=2.75$}
		\label{simvis:e}
	\end{subfigure}
	
	\caption{System configurations for the simulation example. For detailed velocity profiles, cf. Fig. \ref{simplot}. Orange and blue dots represent the approximation for $p(t^{-})$ and $p(t^{+})$, respectively. Red markers represent the ECP used for impulse resolution.}
	\vspace{-1em}
	\label{simvis}
\end{figure*}
\subsection{Cross Complementarities} \label{sec:crosscomp}
To maintain the high-order accuracy of RK methods, we need to ensure that active set changes of (\ref{rkalg:a}) and (\ref{rkalg:c}) do only occur on FE boundaries. To this end, in \cite{fesd_standard} the concept of \textit{cross complementarity} conditions was introduced. To adopt these conditions to the here considered model, we first observe that $z_{n,j}$ contains the discrete time scaling factor $\alpha_{n,j}$ and the ECP $p_{n,j}$. As noted in Section \ref{section1a}, the optimal scaling factor and optimal dual variables of the distance problem (\ref{distprob}) are uniquely determined by the system configuration. Since the latter is continuous in time, it motivates to define $ \alpha_{n,0} = \alpha_{n-1,n_{\mathrm{s}}}$ and $\mu_{n,0} = \mu_{n-1,n_{\mathrm{s}}}$. On the other hand, for the ECP it can happen that we have $p(t_{\mathrm{n}}^{-}) \neq p(t_{\mathrm{n}}^{+})$ and, consequentially, $z(t_{\mathrm{n}}^{-}) \neq z(t_{\mathrm{n}}^{+})$. Given this observation, we now utilize the continuity of $\alpha(t)$ and $\mu(t)$ to fix the active sets in (\ref{rkalg}). The cross complementarity conditions for (\ref{rkalg}) formulate as
\begin{equation} \label{crosscomps}
	\begin{split}
		0 &= \lambda_{\mathrm{n},n,j} (\alpha_{n,j^{\prime}} - 1), \ j = 1,\dots, n_{\mathrm{s}}, \ j^{\prime} = 0,\dots, n_{\mathrm{s}}, \\
		0 &= \mu_{n,j^{\prime}}^{\top}g_{\mathrm{d}}(z_{n,j},q_{n,j}), \ j = 1,\dots, n_{\mathrm{s}}, \ j^{\prime} = 0,\dots, n_{\mathrm{s}}.
	\end{split}
\end{equation}
Including $j^{\prime} = 0$ is a crucial component of the discretization scheme, as this forces $h_{n}$ to adapt such that dynamic switches only occur at FE boundaries. The conditions (\ref{crosscomps}) now ensure that active sets are fixed within a FE, cf. \cite{fesdj,fesd_standard}. The conditions (\ref{crosscomps}) are denoted by $G_{\mathrm{cross}}(\bm{\mathrm{x}},\bm{\mathrm{z}},s_{0}) = 0.$

\subsection{Direct Optimal Control with Exact Switch Detection} \label{sec:discopt}
We compactly denote evaluation of one control interval by
\begin{align*}
s_{1} &= F_{\mathrm{fesd}}(\bm{\mathrm{x}}), \\
0 &= G_{\mathrm{fesd}}(\bm{\mathrm{x}},\bm{\mathrm{z}},\bm{\mathrm{h}},s_{0},\hat{u}, T) = \begin{pmatrix}
G_{\mathrm{rk}}(\bm{\mathrm{x}},\bm{\mathrm{z}},\bm{\mathrm{h}},s_{0},\hat{u}) \\
G_{\mathrm{cross}}(\bm{\mathrm{x}},\bm{\mathrm{z}},s_{0})\\
T - \sum_{n = 0}^{N_{\mathrm{FE}} - 1}h_{n}
\end{pmatrix}, \\
\end{align*}
where $F_{\mathrm{fesd}}(\bm{\mathrm{x}}) = (q_{N_{\mathrm{FE}} - 1, n_{\mathrm{s}}},\nu_{N_{\mathrm{FE}} - 1, n_{\mathrm{s}}})$ evaluates the system state at time $T$ and a condition was added that ensures that all FE sum up to the length of the control interval.
\begin{figure}
	\centering
	\vspace{0.9em}
	\input{figures/simplot.pgf}
	\caption{Evolution of differential and algebraic states for the simulation example. Active set changes of the force complementarity (\ref{cls:c}) and the SDF complementarity (\ref{cls:e}) are marked by purple and orange dotted lines, respectively.}
	\vspace{-1.4em}
	\label{simplot}
\end{figure}

A discrete time OCP for a time-horizon $[0,T]$ split into $N$ equidistant control intervals with corresponding constant controls $\bm{\mathrm{u}} = (\hat{u}_{0},\dots, \hat{u}_{N-1})$ is given by

\begin{mini}
{\bm{\mathrm{s}},\bm{\mathrm{u}}, \mathcal{X}, \mathcal{Z}, \mathcal{H}}{\sum_{k = 0}^{N -1} \hat{L}(s_{k},\bm{\mathrm{x}}_{k},\hat{u}_{k}) + C_{\mathrm{h}}(\bm{\mathrm{h}}_{k}) + L_{\mathrm{t}}(s_{N}),} {\label{ocp}}{}
\addConstraint{s_{0}= \bar{x}_{0},}{}
\addConstraint{s_{k+1}= F_{\mathrm{fesd}}(\bm{\mathrm{x}}_{k}),}{}
\addConstraint{0= G_{\mathrm{fesd}}(\bm{\mathrm{x}}_{k},\bm{\mathrm{z}}_{k},\bm{\mathrm{h}}_{k},s_{k},\hat{u}_{k}, \frac{T}{N}),}{}
\addConstraint{0 \geq G_{\mathrm{p}}(s_{k}, \hat{u}_{k}),}{}
\end{mini}
where $\bm{\mathrm{s}} = (s_{0},\dots,s_{N-1})$, $\mathcal{X} = (\bm{\mathrm{x}}_{0},\dots,\bm{\mathrm{x}}_{N-1})$, $\mathcal{Z} = (\bm{\mathrm{z}}_{0},\dots,\bm{\mathrm{z}}_{N-1})$ and $\mathcal{H} = (\bm{\mathrm{h}}_{0},\dots,\bm{\mathrm{h}}_{N-1})$. The function $\hat{L}: \mathbb{R}^{n_{x}} \times \mathbb{R}^{N_{\mathrm{FE}}(n_{\mathrm{s}} + 1)n_{x}} \times \mathbb{R}^{n_{u}} \rightarrow \mathbb{R}$ is the running cost, $L_{\mathrm{t}}: \mathbb{R}^{n_{x}} \rightarrow \mathbb{R}$ is the terminal cost, $G_{\mathrm{p}}: \mathbb{R}^{n_{x}}\times \mathbb{R}^{n_{u}} \rightarrow \mathbb{R}$ denotes path constraints and $\bar{x}_{0} \in \mathbb{R}^{n_{x}}$ is a given initial state.  The cost function
\begin{equation*}
	C_{\mathrm{h}}(\bm{\mathrm{h}}_{k}) = \sum_{n = 0}^{N_{\mathrm{FE}} - 1}\Vert h_{k,n} - \frac{T}{N \cdot N_{\mathrm{FE}}} \Vert_{2}^{2}
\end{equation*}
restricts spurious DoF corresponding to FE lengths. See \cite{fesd_standard} for an elaborate discussion on step equilibrium.

\begin{figure*}
	\captionsetup[subfigure]{justification=centering}
	\centering
	\begin{subfigure}[t]{0.18\textwidth}
		\centering
\begin{tikzpicture}
\tikzstyle{every node}=[font=\scriptsize]
\definecolor{darkgray176}{RGB}{176,176,176}

\begin{axis}[
tick align=outside,
tick pos=left,
x grid style={darkgray176},
xmin=-6, xmax=8,
xtick style={color=black},
y grid style={darkgray176},
ymin=-4, ymax=7,
width=1.221\textwidth,
height=1.078\textwidth,
ytick style={color=black},
xlabel= {\(\displaystyle c_{\mathrm{x}}\)},
ylabel= {\(\displaystyle c_{\mathrm{y}}\)},
x label style={at={(axis description cs:0.5,-0.23)},anchor=north},
y label style={at={(axis description cs:-0.17,.5)}, anchor=south},
]
\path [draw=mediumcyan]
(axis cs:-2.18205080756888,-0.913397459621557)
--(axis cs:-0.55,1.91339745962156)
.. controls (axis cs:-0.516666666666666,1.97113248654052) and (axis cs:-0.483333333333333,1.97113248654052) .. (axis cs:-0.45,1.91339745962156)
--(axis cs:1.18205080756888,-0.913397459621556)
.. controls (axis cs:1.21538414090221,-0.971132486540518) and (axis cs:1.19871747423554,-0.999999999999999) .. (axis cs:1.13205080756888,-1)
--(axis cs:-2.13205080756888,-1)
.. controls (axis cs:-2.19871747423554,-1) and (axis cs:-2.21538414090221,-0.971132486540519) .. (axis cs:-2.18205080756888,-0.913397459621557);
\path [draw=mediumorange]
(axis cs:2.35244119983534,4.51618795026618)
--(axis cs:3.80201010126777,5.9657568516986)
.. controls (axis cs:3.84915055334687,6.0128973037777) and (axis cs:3.89629100542597,6.0128973037777) .. (axis cs:3.94343145750508,5.9657568516986)
--(axis cs:5.49906637611548,4.4101219330882)
.. controls (axis cs:5.54620682819458,4.36298148100909) and (axis cs:5.56506300902623,4.30641293851417) .. (axis cs:5.55563491861041,4.24041630560343)
--(axis cs:5.3717871555019,2.95348196384391)
.. controls (axis cs:5.36235906508608,2.88748533093317) and (axis cs:5.32507247441408,2.84740602633342) .. (axis cs:5.25992738348589,2.83324405004469)
--(axis cs:3.01584426117686,2.34539989302099)
.. controls (axis cs:2.95069917024868,2.33123791673225) and (axis cs:2.90854836193237,2.35608447142859) .. (axis cs:2.88939183622795,2.41993955711)
--(axis cs:2.31046531027332,4.34969464362541)
.. controls (axis cs:2.2913087845689,4.41354972930682) and (axis cs:2.30530074775624,4.46904749818708) .. (axis cs:2.35244119983534,4.51618795026618);
\path [draw=black, fill=mediumcyan, opacity=0.2]
(axis cs:-6.40523417456881,-1.90961901034025)
--(axis cs:-3.05002604329513,3.90343178846984)
.. controls (axis cs:-3.01669984917006,3.96117093662412) and (axis cs:-2.98336651609155,3.96117505830081) .. (axis cs:-2.9500260440596,3.9034441534999)
--(axis cs:0.406619555562519,-1.90877672257226)
.. controls (axis cs:0.439960027594464,-1.96650762737317) and (axis cs:0.423296930531927,-1.99537720145031) .. (axis cs:0.356630264374907,-1.99538544480368)
--(axis cs:-6.35522346652089,-1.99621536754161)
.. controls (axis cs:-6.42189013267791,-1.99622361089499) and (axis cs:-6.43856036869388,-1.96735815849453) .. (axis cs:-6.40523417456881,-1.90961901034025);
\path [draw=black, fill=mediumorange, opacity=0.2]
(axis cs:6.05277889094544,1.69575228878759)
--(axis cs:6.05277889094137,-2.59362843318781)
.. controls (axis cs:6.05277889094131,-2.66029509985447) and (axis cs:6.01944555760795,-2.69362843318778) .. (axis cs:5.95277889094128,-2.69362843318771)
--(axis cs:1.36410612083419,-2.69362843318336)
.. controls (axis cs:1.29743945416752,-2.6936284331833) and (axis cs:1.24410612083421,-2.66696176651658) .. (axis cs:1.20410612083426,-2.61362843318321)
--(axis cs:-0.471646167953774,-0.379292048128074)
.. controls (axis cs:-0.511646167953724,-0.325958714794703) and (axis cs:-0.513620914909888,-0.271252765615483) .. (axis cs:-0.477570408822269,-0.215174200590413)
--(axis cs:2.10790650610408,3.80667877817553)
.. controls (axis cs:2.1439570121917,3.8627573432006) and (axis cs:2.19133130189957,3.8749932982786) .. (axis cs:2.25002937522769,3.84338664340955)
--(axis cs:5.96473178095337,1.84316227109117)
.. controls (axis cs:6.02342985428148,1.81155561622212) and (axis cs:6.05277889094551,1.76241895545426) .. (axis cs:6.05277889094544,1.69575228878759);
\path [draw=black, fill=mediumcyan]
(axis cs:-4.68148117092414,-0.914125539274943)
--(axis cs:-3.04977990987337,1.91287116193117)
.. controls (axis cs:-3.01645371574829,1.97061031008545) and (axis cs:-2.98312038266978,1.97061443176214) .. (axis cs:-2.94977991063784,1.91288352696123)
--(axis cs:-1.31737958150392,-0.913709567098894)
.. controls (axis cs:-1.28403910947198,-0.971440471899802) and (axis cs:-1.30070220653452,-1.00031004597694) .. (axis cs:-1.36736887269154,-1.00031828933032)
--(axis cs:-4.63147046287622,-1.00072189647631)
.. controls (axis cs:-4.69813712903324,-1.00073013982968) and (axis cs:-4.71480736504921,-0.971864687429226) .. (axis cs:-4.68148117092414,-0.914125539274943);
\path [draw=black, fill=mediumorange]
(axis cs:4.53000000000132,0.799999999998875)
--(axis cs:4.52999999999938,-1.25000000000113)
.. controls (axis cs:4.52999999999932,-1.31666666666779) and (axis cs:4.49666666666595,-1.35000000000109) .. (axis cs:4.42999999999928,-1.35000000000103)
--(axis cs:2.22999999999928,-1.34999999999894)
.. controls (axis cs:2.16333333333262,-1.34999999999888) and (axis cs:2.10999999999931,-1.32333333333216) .. (axis cs:2.06999999999936,-1.26999999999879)
--(axis cs:1.29000000000035,-0.229999999998052)
.. controls (axis cs:1.2500000000004,-0.176666666664681) and (axis cs:1.24802525304423,-0.12196071748546) .. (axis cs:1.28407575913185,-0.065882152460391)
--(axis cs:2.52592424087098,1.86588215246312)
.. controls (axis cs:2.5619747469586,1.92196071748819) and (axis cs:2.60934903666647,1.9341966725662) .. (axis cs:2.66804710999459,1.90259001769714)
--(axis cs:4.44195289000925,0.94740998230246)
.. controls (axis cs:4.50065096333736,0.915803327433403) and (axis cs:4.53000000000139,0.866666666665541) .. (axis cs:4.53000000000132,0.799999999998875);

\addplot [semithick, mediumblue, opacity=0.2, mark=*, mark size=2.2, mark options={solid}]
table {%
-0.496537447572819 -0.344909520087439
};
\addplot [semithick, mediumorange, opacity=0.2, mark=*, mark size=1.3, mark options={solid}]
table {%
-0.496537447572819 -0.344909520087439
};
\draw (axis cs:-4,8) node[
  scale=0.454545454545455,
  anchor=base west,
  text=black,
  rotate=0.0
]{ CntI: 0 FE: 0 ns: 0};
\end{axis}

\end{tikzpicture}
		\vspace{-1.52em}
		\caption{$t=0$}
		\label{cntvis:a}
	\end{subfigure}
	\hspace{0.606em}
	\hfill
	\begin{subfigure}[t]{0.18\textwidth}
		\centering
\begin{tikzpicture}
\tikzstyle{every node}=[font=\scriptsize]

\definecolor{darkgray176}{RGB}{176,176,176}

\begin{axis}[
tick align=outside,
tick pos=left,
x grid style={darkgray176},
xmin=-6, xmax=8,
xtick style={color=black},
y grid style={darkgray176},
ymin=-4, ymax=7,
width=1.221\textwidth,
height=1.078\textwidth,
ytick style={color=black},
yticklabels={,,},
xlabel= {\(\displaystyle c_{\mathrm{x}}\)},
x label style={at={(axis description cs:0.5,-0.23)},anchor=north},
]

\path [draw=mediumcyan]
(axis cs:-2.18205080756888,-0.913397459621557)
--(axis cs:-0.55,1.91339745962156)
.. controls (axis cs:-0.516666666666666,1.97113248654052) and (axis cs:-0.483333333333333,1.97113248654052) .. (axis cs:-0.45,1.91339745962156)
--(axis cs:1.18205080756888,-0.913397459621556)
.. controls (axis cs:1.21538414090221,-0.971132486540518) and (axis cs:1.19871747423554,-0.999999999999999) .. (axis cs:1.13205080756888,-1)
--(axis cs:-2.13205080756888,-1)
.. controls (axis cs:-2.19871747423554,-1) and (axis cs:-2.21538414090221,-0.971132486540519) .. (axis cs:-2.18205080756888,-0.913397459621557);
\path [draw=mediumorange]
(axis cs:2.35244119983534,4.51618795026618)
--(axis cs:3.80201010126777,5.9657568516986)
.. controls (axis cs:3.84915055334687,6.0128973037777) and (axis cs:3.89629100542597,6.0128973037777) .. (axis cs:3.94343145750508,5.9657568516986)
--(axis cs:5.49906637611548,4.4101219330882)
.. controls (axis cs:5.54620682819458,4.36298148100909) and (axis cs:5.56506300902623,4.30641293851417) .. (axis cs:5.55563491861041,4.24041630560343)
--(axis cs:5.3717871555019,2.95348196384391)
.. controls (axis cs:5.36235906508608,2.88748533093317) and (axis cs:5.32507247441408,2.84740602633342) .. (axis cs:5.25992738348589,2.83324405004469)
--(axis cs:3.01584426117686,2.34539989302099)
.. controls (axis cs:2.95069917024868,2.33123791673225) and (axis cs:2.90854836193237,2.35608447142859) .. (axis cs:2.88939183622795,2.41993955711)
--(axis cs:2.31046531027332,4.34969464362541)
.. controls (axis cs:2.2913087845689,4.41354972930682) and (axis cs:2.30530074775624,4.46904749818708) .. (axis cs:2.35244119983534,4.51618795026618);
\path [draw=black, fill=mediumcyan, opacity=0.2]
(axis cs:0.0415991218668247,-3.42112793243751)
--(axis cs:1.32380455360843,-0.419410315705998)
.. controls (axis cs:1.34999257467235,-0.358102632686654) and (axis cs:1.38308658540929,-0.354115457689915) .. (axis cs:1.42308658581924,-0.407448790715781)
--(axis cs:3.38154758102521,-3.01873007584016)
.. controls (axis cs:3.42154758143517,-3.07206340886603) and (axis cs:3.40845357090321,-3.1027172503757) .. (axis cs:3.34226554942933,-3.11069160036918)
--(axis cs:0.10159912248176,-3.50112793197631)
.. controls (axis cs:0.0354111010078864,-3.50910228196979) and (axis cs:0.0154111008029079,-3.48243561545685) .. (axis cs:0.0415991218668247,-3.42112793243751);
\path [draw=black, fill=mediumorange, opacity=0.2]
(axis cs:4.53000001151145,0.799999998064151)
--(axis cs:4.53000000212326,-1.25000000848646)
.. controls (axis cs:4.53000000181795,-1.31666667515313) and (axis cs:4.49666666833197,-1.35000000833381) .. (axis cs:4.4300000016653,-1.3500000080285)
--(axis cs:2.22999999467798,-1.34999999795337)
.. controls (axis cs:2.16333332801131,-1.34999999764807) and (axis cs:2.1099999948001,-1.32333333073716) .. (axis cs:2.06999999504434,-1.26999999722064)
--(axis cs:1.28999999718689,-0.229999990154885)
.. controls (axis cs:1.24999999743113,-0.176666656638367) and (axis cs:1.24802525072545,-0.121960707450105) .. (axis cs:1.28407575706983,-0.0658821425900986)
--(axis cs:2.52592425158422,1.86588216276132)
.. controls (axis cs:2.5619747579286,1.92196072762133) and (axis cs:2.60934904769249,1.93419668248243) .. (axis cs:2.66804712087589,1.90259002734461)
--(axis cs:4.44195290219431,0.947409980770874)
.. controls (axis cs:4.50065097537771,0.915803325633059) and (axis cs:4.53000001181676,0.866666664730818) .. (axis cs:4.53000001151145,0.799999998064151);
\path [draw=black, fill=mediumcyan]
(axis cs:0.0415991265250401,-3.42112792894385)
--(axis cs:1.32380455430493,-0.419410321486961)
.. controls (axis cs:1.34999257536884,-0.358102638467618) and (axis cs:1.38308658610578,-0.354115463470879) .. (axis cs:1.42308658651574,-0.407448796496745)
--(axis cs:3.3815475756705,-3.01873007355286)
.. controls (axis cs:3.42154757608046,-3.07206340657873) and (axis cs:3.4084535655485,-3.1027172480884) .. (axis cs:3.34226554407463,-3.11069159808188)
--(axis cs:0.101599127139975,-3.50112792848265)
.. controls (axis cs:0.0354111056661017,-3.50910227847612) and (axis cs:0.0154111054611233,-3.48243561196319) .. (axis cs:0.0415991265250401,-3.42112792894385);
\path [draw=black, fill=mediumorange]
(axis cs:4.53000000705703,0.799999995443905)
--(axis cs:4.52999999766884,-1.2500000045561)
.. controls (axis cs:4.52999999736353,-1.31666667122276) and (axis cs:4.49666666387755,-1.35000000440344) .. (axis cs:4.42999999721088,-1.35000000409813)
--(axis cs:2.22999999721088,-1.349999994023)
.. controls (axis cs:2.16333333054421,-1.3499999937177) and (axis cs:2.109999997333,-1.32333332680679) .. (axis cs:2.06999999757725,-1.26999999329027)
--(axis cs:1.29000000234004,-0.229999989718177)
.. controls (axis cs:1.25000000258428,-0.176666656201659) and (axis cs:1.2480252558786,-0.121960707013397) .. (axis cs:1.28407576222298,-0.0658821421533909)
--(axis cs:2.525924252807,1.86588215708412)
.. controls (axis cs:2.56197475915138,1.92196072194413) and (axis cs:2.60934904891527,1.93419667680523) .. (axis cs:2.66804712209867,1.90259002166741)
--(axis cs:4.44195289773989,0.947409978150628)
.. controls (axis cs:4.50065097092329,0.915803323012813) and (axis cs:4.53000000736234,0.866666662110572) .. (axis cs:4.53000000705703,0.799999995443905);

\addplot [semithick, mediumred, mark=diamond*, mark size=4.5, mark options={solid}]
table {%
2.09875836808475 -1.30834449142492
};
\addplot [semithick, mediumblue, mark=*, mark size=2.2, mark options={solid}]
table {%
2.09269047196773 -1.30025396246854
};
\addplot [semithick, mediumorange, mark=*, mark size=1.3, mark options={solid}]
table {%
2.09269047196773 -1.30025396246854
};

\draw (axis cs:-4,8) node[
  scale=0.454545454545455,
  anchor=base west,
  text=black,
  rotate=0.0
]{ CntI: 4 FE: 1 ns: 3};
\end{axis}

\end{tikzpicture}
		\vspace{-0.3em}
		\caption{$t=5.00$}
		\label{cntvis:b}
	\end{subfigure}
	\hfill
	\begin{subfigure}[t]{0.18\textwidth}
		\centering
\begin{tikzpicture}
\tikzstyle{every node}=[font=\scriptsize]

\definecolor{darkgray176}{RGB}{176,176,176}

\begin{axis}[
tick align=outside,
tick pos=left,
x grid style={darkgray176},
xmin=-6, xmax=8,
xtick style={color=black},
y grid style={darkgray176},
ymin=-4, ymax=7,
width=1.221\textwidth,
height=1.078\textwidth,
ytick style={color=black},
yticklabels={,,},
xlabel= {\(\displaystyle c_{\mathrm{x}}\)},
x label style={at={(axis description cs:0.5,-0.23)},anchor=north},
y label style={at={(axis description cs:-0.12,.5)}, anchor=south},
]
\path [draw=mediumcyan]
(axis cs:-2.18205080756888,-0.913397459621557)
--(axis cs:-0.55,1.91339745962156)
.. controls (axis cs:-0.516666666666666,1.97113248654052) and (axis cs:-0.483333333333333,1.97113248654052) .. (axis cs:-0.45,1.91339745962156)
--(axis cs:1.18205080756888,-0.913397459621556)
.. controls (axis cs:1.21538414090221,-0.971132486540518) and (axis cs:1.19871747423554,-0.999999999999999) .. (axis cs:1.13205080756888,-1)
--(axis cs:-2.13205080756888,-1)
.. controls (axis cs:-2.19871747423554,-1) and (axis cs:-2.21538414090221,-0.971132486540519) .. (axis cs:-2.18205080756888,-0.913397459621557);
\path [draw=mediumorange]
(axis cs:2.35244119983534,4.51618795026618)
--(axis cs:3.80201010126777,5.9657568516986)
.. controls (axis cs:3.84915055334687,6.0128973037777) and (axis cs:3.89629100542597,6.0128973037777) .. (axis cs:3.94343145750508,5.9657568516986)
--(axis cs:5.49906637611548,4.4101219330882)
.. controls (axis cs:5.54620682819458,4.36298148100909) and (axis cs:5.56506300902623,4.30641293851417) .. (axis cs:5.55563491861041,4.24041630560343)
--(axis cs:5.3717871555019,2.95348196384391)
.. controls (axis cs:5.36235906508608,2.88748533093317) and (axis cs:5.32507247441408,2.84740602633342) .. (axis cs:5.25992738348589,2.83324405004469)
--(axis cs:3.01584426117686,2.34539989302099)
.. controls (axis cs:2.95069917024868,2.33123791673225) and (axis cs:2.90854836193237,2.35608447142859) .. (axis cs:2.88939183622795,2.41993955711)
--(axis cs:2.31046531027332,4.34969464362541)
.. controls (axis cs:2.2913087845689,4.41354972930682) and (axis cs:2.30530074775624,4.46904749818708) .. (axis cs:2.35244119983534,4.51618795026618);
\path [draw=black, fill=mediumcyan, opacity=0.2]
(axis cs:5.02957654413122,-1.64911400039178)
--(axis cs:3.21132893892235,1.06166727371884)
.. controls (axis cs:3.17419268501651,1.11703281435817) and (axis cs:3.18888260388401,1.14695466915992) .. (axis cs:3.25539869552485,1.1514328381241)
--(axis cs:6.51212794561223,1.3706908175499)
.. controls (axis cs:6.57864403725307,1.37516898651407) and (axis cs:6.59721216420599,1.34748621619441) .. (axis cs:6.56783232647099,1.2876425065909)
--(axis cs:5.12935068159248,-1.64239674694551)
.. controls (axis cs:5.09997084385748,-1.70224045654902) and (axis cs:5.06671279803705,-1.70447954103111) .. (axis cs:5.02957654413122,-1.64911400039178);
\path [draw=black, fill=mediumorange, opacity=0.2]
(axis cs:2.977137374282,4.32610267620832)
--(axis cs:5.00334234491575,4.63754041362468)
.. controls (axis cs:5.06923518911833,4.64766847006902) and (axis cs:5.10724563944178,4.6197860761899) .. (axis cs:5.11737369588612,4.55389323198733)
--(axis cs:5.45159956041911,2.37942936113748)
.. controls (axis cs:5.46172761686345,2.31353651693491) and (axis cs:5.44347292433789,2.25677101899511) .. (axis cs:5.39683548284243,2.2091328673181)
--(axis cs:4.48740536829977,1.28018890411955)
.. controls (axis cs:4.44076792680431,1.23255075244254) and (axis cs:4.38699697324055,1.22228795292273) .. (axis cs:4.32609250760847,1.24940050556013)
--(axis cs:2.22808815742969,2.18335916996005)
.. controls (axis cs:2.16718369179762,2.21047172259746) and (axis cs:2.14789262137063,2.25543722630615) .. (axis cs:2.17021494614872,2.31825568108614)
--(axis cs:2.844814620811,4.21668290937182)
.. controls (axis cs:2.86713694558909,4.27950136415182) and (axis cs:2.91124453007943,4.31597461976398) .. (axis cs:2.977137374282,4.32610267620832);
\path [draw=black, fill=mediumcyan]
(axis cs:5.02957654344227,-1.64911399015848)
--(axis cs:3.21132894812912,1.06166726919884)
.. controls (axis cs:3.17419269422328,1.11703280983817) and (axis cs:3.18888261309078,1.14695466463993) .. (axis cs:3.25539870473162,1.1514328336041)
--(axis cs:6.51212793709441,1.3706908118366)
.. controls (axis cs:6.57864402873525,1.37516898080077) and (axis cs:6.59721215568817,1.34748621048111) .. (axis cs:6.56783231795316,1.2876425008776)
--(axis cs:5.12935068090353,-1.64239673671222)
.. controls (axis cs:5.09997084316852,-1.70224044631573) and (axis cs:5.0667127973481,-1.70447953079781) .. (axis cs:5.02957654344227,-1.64911399015848);
\path [draw=black, fill=mediumorange]
(axis cs:2.97713738003584,4.32610266915437)
--(axis cs:5.003342339265,4.63754040481779)
.. controls (axis cs:5.06923518346757,4.64766846126213) and (axis cs:5.10724563379103,4.61978606738302) .. (axis cs:5.11737369023537,4.55389322318044)
--(axis cs:5.45159955289855,2.37942936449549)
.. controls (axis cs:5.46172760934289,2.31353652029292) and (axis cs:5.44347291681733,2.25677102235312) .. (axis cs:5.39683547532188,2.20913287067611)
--(axis cs:4.48740536616048,1.2801889129743)
.. controls (axis cs:4.44076792466503,1.23255076129728) and (axis cs:4.38699697110126,1.22228796177748) .. (axis cs:4.32609250546919,1.24940051441488)
--(axis cs:2.22808816698646,2.18335917360813)
.. controls (axis cs:2.16718370135438,2.21047172624553) and (axis cs:2.14789263092739,2.25543722995423) .. (axis cs:2.17021495570549,2.31825568473422)
--(axis cs:2.84481462656484,4.21668290231788)
.. controls (axis cs:2.86713695134293,4.27950135709787) and (axis cs:2.91124453583326,4.31597461271004) .. (axis cs:2.97713738003584,4.32610266915437);

\addplot [semithick, mediumred, mark=diamond*, mark size=4.5, mark options={solid}]
table {%
4.40133466981716 1.22858251123038
};
\addplot [semithick, mediumblue, mark=*, mark size=2.2, mark options={solid}]
table {%
4.4013346697714 1.22858251140455
};
\addplot [semithick, mediumorange, mark=*, mark size=1.3, mark options={solid}]
table {%
4.4013346697714 1.22858251140455
};

\draw (axis cs:-4,8) node[
  scale=0.454545454545455,
  anchor=base west,
  text=black,
  rotate=0.0
]{ CntI: 12 FE: 0 ns: 3};
\end{axis}

\end{tikzpicture}
		\vspace{-0.3em}
		\caption{$t=12.79$}
		\label{cntvis:c}
	\end{subfigure}
	\hfill
	\begin{subfigure}[t]{0.18\textwidth}
		\centering
\begin{tikzpicture}
\tikzstyle{every node}=[font=\scriptsize]
\definecolor{darkgray176}{RGB}{176,176,176}

\begin{axis}[
tick align=outside,
tick pos=left,
x grid style={darkgray176},
xmin=-6, xmax=8,
xtick style={color=black},
y grid style={darkgray176},
ymin=-4, ymax=7,
width=1.221\textwidth,
height=1.078\textwidth,
ytick style={color=black},
yticklabels={,,},
xlabel= {\(\displaystyle c_{\mathrm{x}}\)},
x label style={at={(axis description cs:0.5,-0.23)},anchor=north},
]
\path [draw=mediumcyan]
(axis cs:-2.18205080756888,-0.913397459621557)
--(axis cs:-0.55,1.91339745962156)
.. controls (axis cs:-0.516666666666666,1.97113248654052) and (axis cs:-0.483333333333333,1.97113248654052) .. (axis cs:-0.45,1.91339745962156)
--(axis cs:1.18205080756888,-0.913397459621556)
.. controls (axis cs:1.21538414090221,-0.971132486540518) and (axis cs:1.19871747423554,-0.999999999999999) .. (axis cs:1.13205080756888,-1)
--(axis cs:-2.13205080756888,-1)
.. controls (axis cs:-2.19871747423554,-1) and (axis cs:-2.21538414090221,-0.971132486540519) .. (axis cs:-2.18205080756888,-0.913397459621557);
\path [draw=mediumorange]
(axis cs:2.35244119983534,4.51618795026618)
--(axis cs:3.80201010126777,5.9657568516986)
.. controls (axis cs:3.84915055334687,6.0128973037777) and (axis cs:3.89629100542597,6.0128973037777) .. (axis cs:3.94343145750508,5.9657568516986)
--(axis cs:5.49906637611548,4.4101219330882)
.. controls (axis cs:5.54620682819458,4.36298148100909) and (axis cs:5.56506300902623,4.30641293851417) .. (axis cs:5.55563491861041,4.24041630560343)
--(axis cs:5.3717871555019,2.95348196384391)
.. controls (axis cs:5.36235906508608,2.88748533093317) and (axis cs:5.32507247441408,2.84740602633342) .. (axis cs:5.25992738348589,2.83324405004469)
--(axis cs:3.01584426117686,2.34539989302099)
.. controls (axis cs:2.95069917024868,2.33123791673225) and (axis cs:2.90854836193237,2.35608447142859) .. (axis cs:2.88939183622795,2.41993955711)
--(axis cs:2.31046531027332,4.34969464362541)
.. controls (axis cs:2.2913087845689,4.41354972930682) and (axis cs:2.30530074775624,4.46904749818708) .. (axis cs:2.35244119983534,4.51618795026618);
\path [draw=black, fill=mediumcyan, opacity=0.2]
(axis cs:3.6831655620857,-0.93805961648615)
--(axis cs:1.73744078994373,1.68272548272063)
.. controls (axis cs:1.69770091671967,1.73625292353916) and (axis cs:1.71094401018786,1.76684265362139) .. (axis cs:1.7771700703483,1.77449467296733)
--(axis cs:5.01969893019208,2.14914920481158)
.. controls (axis cs:5.08592499035252,2.15680122415752) and (axis cs:5.10579492696455,2.13003750374826) .. (axis cs:5.07930874002817,2.06885804358379)
--(axis cs:3.78250465232636,-0.926581587467238)
.. controls (axis cs:3.75601846538998,-0.987761047631703) and (axis cs:3.72290543530976,-0.991587057304674) .. (axis cs:3.6831655620857,-0.93805961648615);
\path [draw=black, fill=mediumorange, opacity=0.2]
(axis cs:2.56105832471874,4.55476921223305)
--(axis cs:4.35122313752646,5.55367359806389)
.. controls (axis cs:4.40943987920519,5.5861582934258) and (axis cs:4.45479059772551,5.57329227026738) .. (axis cs:4.48727529308742,5.51507552858865)
--(axis cs:5.5592702437127,3.59392304659134)
.. controls (axis cs:5.5917549390746,3.5357063049126) and (axis cs:5.59445599869264,3.47613903342485) .. (axis cs:5.5673734225668,3.41522123212809)
--(axis cs:5.03926318619419,2.22732410252532)
.. controls (axis cs:5.01218061006835,2.16640630122855) and (axis cs:4.96537081419284,2.13802525982352) .. (axis cs:4.89883379856768,2.14218097831023)
--(axis cs:2.60680239363891,2.28533493843433)
.. controls (axis cs:2.54026537801375,2.28949065692104) and (axis cs:2.50649622593698,2.32489808961549) .. (axis cs:2.5054949374086,2.39155723651767)
--(axis cs:2.4752351449932,4.40605344883691)
.. controls (axis cs:2.47423385646482,4.47271259573909) and (axis cs:2.50284158304,4.52228451687114) .. (axis cs:2.56105832471874,4.55476921223305);
\path [draw=black, fill=mediumcyan]
(axis cs:3.68316556136288,-0.938059610230323)
--(axis cs:1.73744079572284,1.6827254802187)
.. controls (axis cs:1.69770092249879,1.73625292103723) and (axis cs:1.71094401596698,1.76684265111946) .. (axis cs:1.77717007612742,1.7744946704654)
--(axis cs:5.01969892513579,2.14914920105768)
.. controls (axis cs:5.08592498529623,2.15680122040363) and (axis cs:5.10579492190826,2.13003749999436) .. (axis cs:5.07930873497187,2.0688580398299)
--(axis cs:3.78250465160354,-0.926581581211411)
.. controls (axis cs:3.75601846466715,-0.987761041375877) and (axis cs:3.72290543458694,-0.991587051048847) .. (axis cs:3.68316556136288,-0.938059610230323);
\path [draw=black, fill=mediumorange]
(axis cs:2.56105832954085,4.55476920940697)
--(axis cs:4.35122313616191,5.55367359178568)
.. controls (axis cs:4.40943987784064,5.58615828714759) and (axis cs:4.45479059636096,5.57329226398918) .. (axis cs:4.48727529172287,5.51507552231044)
--(axis cs:5.55927023866588,3.59392304691223)
.. controls (axis cs:5.59175493402778,3.5357063052335) and (axis cs:5.59445599364582,3.47613903374575) .. (axis cs:5.56737341751998,3.41522123244898)
--(axis cs:5.03926318306607,2.22732410716202)
.. controls (axis cs:5.01218060694023,2.16640630586525) and (axis cs:4.96537081106472,2.13802526446022) .. (axis cs:4.89883379543956,2.14218098294693)
--(axis cs:2.60680239835629,2.28533494258102)
.. controls (axis cs:2.54026538273112,2.28949066106773) and (axis cs:2.50649623065435,2.32489809376218) .. (axis cs:2.50549494212598,2.39155724066436)
--(axis cs:2.47523514981531,4.40605344601083)
.. controls (axis cs:2.47423386128694,4.47271259291302) and (axis cs:2.50284158786212,4.52228451404506) .. (axis cs:2.56105832954085,4.55476920940697);

\addplot [semithick, mediumblue, mark=*, mark size=2.2, mark options={solid}]
table {%
4.94696758090566 2.14074553576263
};
\addplot [semithick, mediumorange, mark=*, mark size=1.3, mark options={solid}]
table {%
4.94696758090566 2.14074553576263
};

\draw (axis cs:-4,8) node[
  scale=0.454545454545455,
  anchor=base west,
  text=black,
  rotate=0.0
]{ CntI: 14 FE: 0 ns: 3};
\end{axis}

\end{tikzpicture}
		\vspace{-0.3em}
		\caption{$t=14.78$}
		\label{cntvis:d}
	\end{subfigure}
	\hfill
	\begin{subfigure}[t]{0.18\textwidth}
		\centering
\begin{tikzpicture}
\tikzstyle{every node}=[font=\scriptsize]
\definecolor{darkgray176}{RGB}{176,176,176}

\begin{axis}[
tick align=outside,
tick pos=left,
x grid style={darkgray176},
xmin=-6, xmax=8,
xtick style={color=black},
y grid style={darkgray176},
ymin=-4, ymax=7,
width=1.221\textwidth,
height=1.078\textwidth,
ytick style={color=black},
yticklabels={,,},
xlabel= {\(\displaystyle c_{\mathrm{x}}\)},
x label style={at={(axis description cs:0.5,-0.23)},anchor=north},
]
\path [draw=mediumcyan]
(axis cs:-2.18205080756888,-0.913397459621557)
--(axis cs:-0.55,1.91339745962156)
.. controls (axis cs:-0.516666666666666,1.97113248654052) and (axis cs:-0.483333333333333,1.97113248654052) .. (axis cs:-0.45,1.91339745962156)
--(axis cs:1.18205080756888,-0.913397459621556)
.. controls (axis cs:1.21538414090221,-0.971132486540518) and (axis cs:1.19871747423554,-0.999999999999999) .. (axis cs:1.13205080756888,-1)
--(axis cs:-2.13205080756888,-1)
.. controls (axis cs:-2.19871747423554,-1) and (axis cs:-2.21538414090221,-0.971132486540519) .. (axis cs:-2.18205080756888,-0.913397459621557);
\path [draw=mediumorange]
(axis cs:2.35244119983534,4.51618795026618)
--(axis cs:3.80201010126777,5.9657568516986)
.. controls (axis cs:3.84915055334687,6.0128973037777) and (axis cs:3.89629100542597,6.0128973037777) .. (axis cs:3.94343145750508,5.9657568516986)
--(axis cs:5.49906637611548,4.4101219330882)
.. controls (axis cs:5.54620682819458,4.36298148100909) and (axis cs:5.56506300902623,4.30641293851417) .. (axis cs:5.55563491861041,4.24041630560343)
--(axis cs:5.3717871555019,2.95348196384391)
.. controls (axis cs:5.36235906508608,2.88748533093317) and (axis cs:5.32507247441408,2.84740602633342) .. (axis cs:5.25992738348589,2.83324405004469)
--(axis cs:3.01584426117686,2.34539989302099)
.. controls (axis cs:2.95069917024868,2.33123791673225) and (axis cs:2.90854836193237,2.35608447142859) .. (axis cs:2.88939183622795,2.41993955711)
--(axis cs:2.31046531027332,4.34969464362541)
.. controls (axis cs:2.2913087845689,4.41354972930682) and (axis cs:2.30530074775624,4.46904749818708) .. (axis cs:2.35244119983534,4.51618795026618);
\path [draw=black, fill=mediumcyan, opacity=0.2]
(axis cs:-4.14048227280132,-2.04618714946622)
--(axis cs:-0.552250645816139,4.17726183847685)
.. controls (axis cs:-0.518951271960502,4.23501645859739) and (axis cs:-0.485617944391442,4.23503606178297) .. (axis cs:-0.452250663108961,4.17732064803361)
--(axis cs:3.14329844611371,-2.04190358956987)
.. controls (axis cs:3.17666572739619,-2.09961900331923) and (axis cs:3.16001604046837,-2.12849631337949) .. (axis cs:3.09334938533025,-2.12853551975066)
--(axis cs:-4.0904313508776,-2.13276027009027)
.. controls (axis cs:-4.15709800601572,-2.13279947646144) and (axis cs:-4.17378164665696,-2.10394176958676) .. (axis cs:-4.14048227280132,-2.04618714946622);
\path [draw=black, fill=mediumorange, opacity=0.2]
(axis cs:0.420729267514523,5.03828304196311)
--(axis cs:3.68475479272701,8.27378906598271)
.. controls (axis cs:3.73210163809042,8.32072221715091) and (axis cs:3.77924163635623,8.3205153700533) .. (axis cs:3.82617478752443,8.27316852468988)
--(axis cs:7.28676784337898,4.78207192872389)
.. controls (axis cs:7.33370099454718,4.73472508336047) and (axis cs:7.35230877733637,4.67807434660246) .. (axis cs:7.34259119174656,4.61211971844985)
--(axis cs:6.90569757349756,1.64686092370701)
.. controls (axis cs:6.89597998790774,1.5809062955544) and (axis cs:6.85851789262526,1.54099098623346) .. (axis cs:6.7933112876501,1.52711499574419)
--(axis cs:1.78415282551183,0.461164158304382)
.. controls (axis cs:1.71894622053668,0.447288167815112) and (axis cs:1.67690484193719,0.472319436409623) .. (axis cs:1.65802868971337,0.536257964087916)
--(axis cs:0.378023227805122,4.87197552369337)
.. controls (axis cs:0.359147075581306,4.93591405137166) and (axis cs:0.373382422151106,4.99134989079491) .. (axis cs:0.420729267514523,5.03828304196311);
\path [draw=black, fill=mediumcyan]
(axis cs:-2.18130786855746,-0.91352051535474)
--(axis cs:-0.550919768778514,1.91423371690168)
.. controls (axis cs:-0.517620394922876,1.97198833702221) and (axis cs:-0.484287067353816,1.9720079402078) .. (axis cs:-0.450919786071335,1.91429252645843)
--(axis cs:1.18279316483222,-0.911542102106178)
.. controls (axis cs:1.2161604461147,-0.969257515855541) and (axis cs:1.19951075918688,-0.998134825915807) .. (axis cs:1.13284410404876,-0.998174032286975)
--(axis cs:-2.13125694663373,-1.00009363597879)
.. controls (axis cs:-2.19792360177185,-1.00013284234995) and (axis cs:-2.21460724241309,-0.971275135475272) .. (axis cs:-2.18130786855746,-0.91352051535474);
\path [draw=black, fill=mediumorange]
(axis cs:2.36274518503575,4.52569287160665)
--(axis cs:3.81866067996083,5.9688872700288)
.. controls (axis cs:3.86600752532425,6.015820421197) and (axis cs:3.91314752359006,6.01561357409939) .. (axis cs:3.96008067475826,5.96826672873598)
--(axis cs:5.50887466330885,4.40582083174321)
.. controls (axis cs:5.55580781447705,4.35847398637979) and (axis cs:5.57441559726625,4.30182324962178) .. (axis cs:5.56469801167643,4.23586862146917)
--(axis cs:5.37520509267507,2.94975337249327)
.. controls (axis cs:5.36548750708525,2.88379874434066) and (axis cs:5.32802541180277,2.84388343501972) .. (axis cs:5.26281880682761,2.83000744453045)
--(axis cs:3.01661668164939,2.35201477280916)
.. controls (axis cs:2.95141007667424,2.33813878231989) and (axis cs:2.90936869807475,2.3631700509144) .. (axis cs:2.89049254585093,2.42710857859269)
--(axis cs:2.32003914532635,4.35938535333691)
.. controls (axis cs:2.30116299310254,4.42332388101521) and (axis cs:2.31539833967234,4.47875972043845) .. (axis cs:2.36274518503575,4.52569287160665);

\addplot [semithick, mediumblue, opacity=0.2, mark=*, mark size=2.2, mark options={solid}]
table {%
1.66947414419505 0.507371503867502
};
\addplot [semithick, mediumorange, opacity=0.2, mark=*, mark size=1.3, mark options={solid}]
table {%
1.66947414419505 0.507371503867502
};

\draw (axis cs:-4,8) node[
  scale=0.454545454545455,
  anchor=base west,
  text=black,
  rotate=0.0
]{ CntI: 19 FE: 1 ns: 3};
\end{axis}

\end{tikzpicture}
		\vspace{-0.3em}
		\caption{$t=20.00$}
		\label{cntvis:e}
	\end{subfigure}
	
	\caption{System configurations for the optimal control example. For detailed velocity profiles, cf. Fig. \ref{cntplot}. Orange and blue dots represent the approximation for $p(t^{-})$ and $p(t^{+})$, respectively. Red markers represent the ECP used for impulse resolution.}
	\vspace{-1em}
	\label{cntvis}
\end{figure*}
\section{NUMERICAL EXAMPLES} \label{secnumex}
In the following, we consider two numerical examples, which illustrate the different dynamic switches permitted by the CLS (\ref{cls}). To this end, we consider the simulation of a falling cuboid in Section \ref{sec4:sim}. In Section \ref{sec4:cnt} the optimal control of a planar manipulation task is discussed.

Due to the complementarity constraints in the FESD-J discretization, the discrete time OCP (\ref{ocp}) is a mathematical problem with complementarity constraints (MPCC). Here we solve MPCC via a relaxation strategy \cite{scholtes} using \texttt{IPOPT} \cite{ipopt} as solver, which is called from the \texttt{CasADi} \cite{casadi} interface. 

\subsection{Simulation of a Falling Cuboid} \label{sec4:sim}
We consider a single cuboid with DoF making contact with another fixed cuboid. Thus, there are three DoF for the system configuration $q = (c_{\mathrm{x}},c_{\mathrm{y}},\xi) \in \mathbb{R}^{3}$. The cuboid with DoF is under the influence of a constant external force $f_{\mathrm{v}}(q,\nu,u) = (0,-1,0)^{\top}$ and has an initial horizontal velocity $\hat{\nu}_{0} = (1,0,0)$. The time horizon $[0,T]$, with $T = 3.68$, is divided into $N = 23$ simulation steps, for each $N_{\mathrm{FE}} = 2$ FE are considered. The Radau IIA method with $n_{\mathrm{s}} = 4$ stage points is used, which has seventh-order accuracy \cite{hairer}. 

Fig. \ref{simvis} shows the configuration of the falling cuboid at time instances where dynamic switches occur. Fig. \ref{simplot} visualizes the evolution of the state trajectories as well as force and impulse magnitudes and the length of the FE. We observe that the FE lengths are practically equidistant for the simulation steps where no switches happen, and, on the other hand, adjust to accurately detect switches if one occurs.

At the state in Fig. \ref{simvis:b}, the bodies make contact for the first time and an impact occurs. Consequently, system velocities undergo a jump. At the time instance in Fig. \ref{simvis:c} another impact occurs, as the cuboids form a patch contact. At the same time the active set of the SDF changes and the ECP parameterizing the contact force normal jumps, which illustrates that it is generally discontinuous. We observe in Fig. \ref{simvis:c} that the ECP for the impulse simulation is chosen within the contact surface such that patch contact is maintained. Due to the initial velocity and no involved friction the cuboid then slides horizontally until the active set of the SDF changes at the configuration in Fig. \ref{simvis:d}. Another active set change in the SDF occurs in Fig. \ref{simvis:e}, which simultaneously results in a jump of the \mbox{force magnitude}.

\begin{figure}
	\centering
	\vspace{1.8em}
	\input{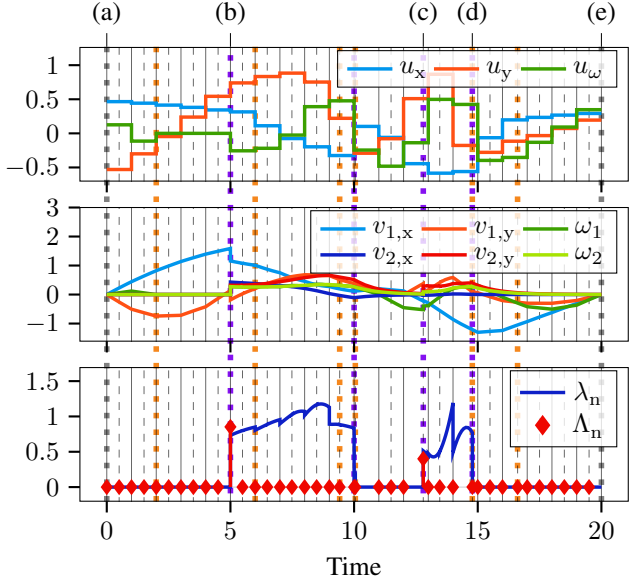}
	\caption{Evolution of differential and algebraic states for the optimal control example. Active set changes of the force complementarity (\ref{cls:c}) and the SDF complementarity (\ref{cls:e}) are marked by purple and orange dotted lines, respectively. }
	\vspace{-1em}
	\label{cntplot}
\end{figure}
\subsection{Optimal Control for a Manipulation Task} \label{sec4:cnt}
We now consider a manipulation task, where a triangular shaped object can be actuated and through contact force can influence another object, which is a nonregular pentagon. Both objects should end up in a goal position, which is transcribed in the time discrete OCP (\ref{ocp}) through the terminal cost function. The nominal dynamics of the CLS (\ref{cls}) are given by $f_{\mathrm{v}}(q,v,u) = (u,\bm{0}) + (-0.1\nu_{1},-\nu_{2})$, where $u \in \mathbb{R}^{3}$, and $\nu_{1} = (v_{1,\mathrm{x}},v_{1,\mathrm{y}},\omega_{1})\in \mathbb{R}^{3}$, $\nu_{2} = (v_{2,\mathrm{x}},v_{2,\mathrm{y}},\omega_{2})\in \mathbb{R}^{3}$, $\nu = (\nu_{1},\nu_{2})$, are the translational and angular velocities of the two bodies. Including forces proportional to the negative velocity models simplified friction effects. We consider the time horizon $[0,T]$, with $T = 20$, which is divided into $N=20$ equidistant control intervals. For each control interval $N_{\mathrm{FE}} = 2$ FE are used together with the seventh-order Radau IIA method, i.e., $n_{\mathrm{s}} = 4$.

Fig. \ref{cntvis} displays five system configurations for a solution trajectory of the discrete time OCP (\ref{ocp}), and Fig. \ref{cntplot} visualizes the corresponding control, velocity, force and impulse magnitudes. We again note that the FE lengths adapt to detect switches if any occur. Furthermore, we can observe that since control forces change discontinuously across control intervals, contact forces are also discontinuous. However, we observe that within an equidistant control interval, contact forces remain continuous if no switch occurs. This is the expected behavior, as if external control forces are smooth and no dynamic switch occurs, CLS (\ref{cls}) evolves smoothly.


\section{CONCLUSION AND OUTLOOK}
This paper presented a complementarity Lagrangian system with an embedded nondifferentiable signed distance function that allows one to model dynamic interactions of rigid bodies with patch contacts. We discussed that if a suitable impact law is chosen, the evolution of state trajectories for this system is unique. We then showed that the FESD-J method can be extended to the considered system such that the integration order of an underlying Runge-Kutta method is maintained. In further work, the methodology will be extended to incorporate frictional effects and will be discussed in a detailed manner in a journal publication.


\section{APPENDIX}
\setcounter{proposition}{0}

\subsection{Proof of Proposition \ref{uniquepdvar}} \label{app:proof1}
We investigate the KKT conditions of the distance problem (\ref{distprob}) to prove the proposition. To this end, let $(z, \mu)$ be optimal for (\ref{distprob}). We define
\begin{align*}
	&\mathcal{A}_{1,\mathrm{p}} = \{i \ \vert \ \mu_{1,\mathrm{p},i} > 0\}, \\
	&\mathcal{A}_{2,\mathrm{p}} = \{i \ \vert \ \mu_{2,\mathrm{p},i} > 0\},
\end{align*}
as the active sets for the constraints (\ref{distprob:b}) and (\ref{distprob:c}), respectively. We first observe that in two dimensions it holds
\begin{equation*}
1 \leq \vert \mathcal{A}_{1,\mathrm{p}} \vert \leq 2, \qquad 1 \leq \vert \mathcal{A}_{2,\mathrm{p}} \vert \leq 2.
\end{equation*}
The constraints (\ref{distprob:c}) and (\ref{distprob:d}) are necessarily active as well. If any of the aforementioned active constraints would be inactive, this would imply that $\alpha$ can be decreased such that there are primal variables $p$, $y_{1}$, $y_{2}$, fulfilling the constraints (\ref{distprob:a})-(\ref{distprob:d}), which would contradict optimality. In the following, we distinguish three cases.

Case 1: Let $\vert \mathcal{A}_{1,\mathrm{p}} \vert = 2$ and $\vert \mathcal{A}_{2,\mathrm{p}} \vert = 2$. Then we have $\mathcal{A}_{1,\mathrm{p}} = \{k_{1,1},k_{1,2}\}$, $\mathcal{A}_{2,\mathrm{p}} = \{k_{2,1},k_{2,2}\}$, $k_{1,1},k_{1,2}\in \{1,\dots,\nha\}$, $k_{2,1},k_{2,2}\in \{1,\dots,\nhb\}$. We select active components of the halfspace representation and dual variables as
\begin{alignat*}{2}
&\bar{A}_{1} = A_{1,\mathcal{A}_{1,\mathrm{p}}}, &&\bar{A}_{2} = A_{2,\mathcal{A}_{2,\mathrm{p}}}, \\
&\bar{b}_{1} = b_{1,\mathcal{A}_{1,\mathrm{p}}},  &&\bar{b}_{2} = b_{2,\mathcal{A}_{2,\mathrm{p}}}, \\
&\bar{\mu}_{1,\mathrm{p}} = \mu_{1, \mathrm{p},\mathcal{A}_{1,\mathrm{p}}},  \quad &&\bar{\mu}_{2,\mathrm{p}} = \mu_{2, \mathrm{p},\mathcal{A}_{2,\mathrm{p}}},
\end{alignat*}
where $\bar{A}_{1} \in \mathbb{R}^{2\times2}$ and $\bar{A}_{2} \in \mathbb{R}^{2\times2}$ contain the normal vectors corresponding to active halfspaces as rows.

Then the equality conditions imposed by the active inequality constraints of (\ref{distprob}) together with the equality conditions imposed by the stationary KKT condition (\ref{statKKTcond}) are given by
\begin{subequations} \label{KKTsystem}
	\begin{align}
	\bar{A}_{1}R( \xi_{1})^{\top}(y_{1}-c_{1})&= \alpha \bar{b}_{1} - \bm{1}r, \label{KKTsystem:a}\\
		\bar{A}_{2}R( \xi_{2})^{\top}(y_{2}-c_{2})&= \alpha \bar{b}_{2} - \bm{1}r, \label{KKTsystem:b}\\
		\Vert p - y_{1} \Vert_{2}^{2} &= r^2, \label{KKTsystem:c}\\
		\Vert p - y_{2} \Vert_{2}^{2} &= r^2, \label{KKTsystem:d}\\
		1 -\bar{b}_{1}^{\top} \bar{\mu}_{1,\mathrm{p}} -\bar{b}_{2}^{\top} \bar{\mu}_{2,\mathrm{p}} &= 0, \label{KKTsystem:e}\\
		2(p-y_{1})\mu_{1, \mathrm{c}} &= - 2(p-y_{2})\mu_{2, \mathrm{c}},  \label{KKTsystem:f}\\
		R( \xi_{1})\bar{A}_{1}^{\top}\bar{\mu}_{1,\mathrm{p}} &=  2(p-y_{1}) \mu_{1, \mathrm{c}}, \label{KKTsystem:g}\\
		R( \xi_{2})\bar{A}_{2}^{\top}\bar{\mu}_{2,\mathrm{p}} &=  2(p-y_{2}) \mu_{2, \mathrm{c}}.\label{KKTsystem:h}
	\end{align}
\end{subequations}
Resolving (\ref{KKTsystem:a}) and (\ref{KKTsystem:b}) for $y_1$ and $y_2$, respectively, results in
\begin{subequations}\label{yresolve}
\begin{align} 
y_{1} &= c_{1} + R(\xi_{1})\bar{A}_{1}^{-1}(\alpha \bar{b}_{1} - \bm{1}r), \\
y_{2} &= c_{2} + R(\xi_{2})\bar{A}_{2}^{-1}(\alpha \bar{b}_{2} - \bm{1}r).
\end{align}
\end{subequations}
Due to the assumptions in Section \ref{section1a} that the considered polytopes are not degenerate, the matrices $\bar{A}_{1}$ and $\bar{A}_{2}$ are invertible. Condition (\ref{KKTsystem:f}) implies that $p$ is element of the straight line connecting $y_{1}$ and $y_{2}$. Conditions (\ref{KKTsystem:c}) and (\ref{KKTsystem:d}) imply that $p$ has a distance of $r$ to $y_{1}$ and $y_{2}$. Thus, we obtain
\begin{equation}  \label{ydisteq}
\Vert y_{1}-y_{2} \Vert_{2} = 2r,
\end{equation}
and
\begin{equation} \label{pexpl}
p = y_{1} + \frac{y_{2} - y_{1}}{2}.
\end{equation}
Substituting (\ref{yresolve}) in (\ref{ydisteq}) results in the condition
\begin{equation} \label{quadalphaeq}
	\begin{split}
	&\Vert c_{1}-c_{2} + (R(\xi_{2})\bar{A}_{2}^{-1}-R(\xi_{1})\bar{A}_{1}^{-1})\bm{1}r \\
	&\quad \quad + (R(\xi_{1})\bar{A}_{1}^{-1}\bar{b}_{1}-R(\xi_{2})\bar{A}_{2}^{-1}\bar{b}_{2}) \alpha \Vert_{2} = 2r.
	\end{split}
\end{equation}
This is a quadratic equation for the scalar $\alpha$ that, due to the geometric interpretation of the problem, uniquely fixes $\alpha$. In particular, the solution of (\ref{quadalphaeq}) with larger value violates the constraints (\ref{KKTsystem:h})-(\ref{KKTsystem:g}). Consequently, we derived unique expressions for the optimal primal variables. Finally, for the dual variables, we observe that (\ref{KKTsystem:c}), (\ref{KKTsystem:d}) and (\ref{KKTsystem:f}) imply
\begin{equation} \label{eqofmuc}
\mu_{1, \mathrm{c}} = \mu_{2, \mathrm{c}}.
\end{equation}
Solving (\ref{KKTsystem:g}) and (\ref{KKTsystem:h}) for $\bar{\mu}_{1,\mathrm{p}}$ and $\bar{\mu}_{2,\mathrm{p}}$, respectively, yields
\begin{subequations} \label{mupeq}
\begin{align}
\bar{\mu}_{1,\mathrm{p}} &= 2(\bar{A}_{1}^{\top})^{-1}R( \xi_{1})^{\top}(p-y_{1}) \mu_{1, \mathrm{c}}, \\
 \bar{\mu}_{2,\mathrm{p}} &= 2(\bar{A}_{2}^{\top})^{-1}R( \xi_{2})^{\top}(p-y_{2}) \mu_{2, \mathrm{c}}.
\end{align}
\end{subequations}
Substituting the above in (\ref{KKTsystem:e}) and using (\ref{KKTsystem:f}) yields
\begin{equation*}
1 -2(\bar{b}_{1}^{\top} (\bar{A}_{1}^{\top})^{-1}R( \xi_{1})^{\top}-\bar{b}_{2}^{\top}(\bar{A}_{2}^{\top})^{-1}R( \xi_{2})^{\top})(p-y_{1}) \mu_{1, \mathrm{c}},
\end{equation*}
which lets us solve for $\mu_{1, \mathrm{c}}$ since primal variables already have been uniquely determined. Consequently, due to (\ref{eqofmuc}) we can solve for $\mu_{2, \mathrm{c}}$, and due to (\ref{mupeq}) for $\mu_{1, \mathrm{p}}$ and $\mu_{2, \mathrm{p}}$.

Case 2:  Here we consider $\vert \mathcal{A}_{1,\mathrm{p}} \vert = 2$ and $\vert \mathcal{A}_{2,\mathrm{p}} \vert = 1$. We redefine
\begin{equation*}
\bar{A}_{2} = A_{2,k_{2,1}}, \quad
\bar{b}_{2} = b_{2,k_{2,1}}, \quad
\bar{\mu}_{2,\mathrm{p}} = \mu_{2, \mathrm{p},k_{2,1}},
\end{equation*}
i.e., here $\bar{A}_{2} \in  \mathbb{R}^{1 \times 2}$ is a row vector. Elaborate discussion can then be carried out analogously to Case 1. Here, (\ref{KKTsystem:b}) does not directly determine $y_{2}$ for a given $\alpha$ as it is only one scalar equation. However, $y_{1}$ is again uniquely determined by (\ref{KKTsystem:a}) and this time, (\ref{KKTsystem:h}) fixes the slope of the straight line passing through $y_{1}$, $y_{2}$, $p$. Given unique primal variables, derivation of the dual variables follows similarly as in \mbox{Case 1}.
The vice versa case with $\vert \mathcal{A}_{1,\mathrm{p}} \vert = 1$ and $\vert \mathcal{A}_{2,\mathrm{p}} \vert = 2$ follows analogously.

Case 3: This is the case of possible patch contact with $\vert \mathcal{A}_{1,\mathrm{p}} \vert = 1$ and $\vert \mathcal{A}_{2,\mathrm{p}} \vert = 1$. Here we have
\begin{alignat*}{2}
	&\bar{A}_{1} = A_{1,k_{1,1}}, &&\bar{A}_{2} = A_{2,k_{2,1}}, \\
	&\bar{b}_{1} = b_{1,k_{1,1}},  &&\bar{b}_{2} = b_{2,k_{2,1}}, \\
	&\bar{\mu}_{1,\mathrm{p}} = \mu_{1, \mathrm{p},k_{1,1}},  \quad &&\bar{\mu}_{2,\mathrm{p}} = \mu_{2, \mathrm{p},k_{2,1}},
\end{alignat*}
where now $\bar{A}_{1} \in \mathbb{R}^{1\times2}$ and $\bar{A}_{2}\in \mathbb{R}^{1\times2}$ are row vectors. In this case, (\ref{KKTsystem:g}) and (\ref{KKTsystem:h}) imply that the vectors $p-y_{i}$ are positive multiples of the vectors $R( \xi_{i})\bar{A}_{i}^{\top}$, for $i =1,2,$ respectively. Furthermore, (\ref{KKTsystem:c}) and (\ref{KKTsystem:d}) imply that \mbox{$p-y_{i}$} have an Euclidean norm equal to $r$. Therefore, we can equivalently formulate the system (\ref{KKTsystem:e})-(\ref{KKTsystem:h}) regrading the dual variables as
\begin{subequations} \label{refdualsystem}
\begin{align}
		1 -\bar{b}_{1}^{\top} \bar{\mu}_{1,\mathrm{p}} -\bar{b}_{2}^{\top} \bar{\mu}_{2,\mathrm{p}} &= 0, \label{refdualsystem:e}\\
2rR( \xi_{1})\bar{A}_{1}^{\top}\mu_{1, \mathrm{c}} &= - 2rR( \xi_{2})\bar{A}_{2}^{\top}\mu_{2, \mathrm{c}},  \label{refdualsystem:f}\\
R( \xi_{1})\bar{A}_{1}^{\top}\bar{\mu}_{1,\mathrm{p}} &=  2rR( \xi_{1})\bar{A}_{1}^{\top}  \mu_{1, \mathrm{c}}, \label{refdualsystem:g}\\
R( \xi_{2})\bar{A}_{2}^{\top}\bar{\mu}_{2,\mathrm{p}} &=  2rR( \xi_{2})\bar{A}_{2}^{\top}  \mu_{2, \mathrm{c}},\label{refdualsystem:h}
\end{align}
\end{subequations}
where we used that $A_{1}$ and $A_{2}$ are assumed to be constructed such that their rows are vectors with unitary norm. System (\ref{refdualsystem}) can be restated as
\begin{subequations} \label{refdualsystemcondensed}
	\begin{align}
		1 -\bar{b}_{1}^{\top} \bar{\mu}_{1,\mathrm{p}} -\bar{b}_{2}^{\top} \bar{\mu}_{2,\mathrm{p}} &= 0, \\
		\mu_{1, \mathrm{c}} &=  \mu_{2, \mathrm{c}},  \\
	\bar{\mu}_{1,\mathrm{p}} &=  2r \mu_{1, \mathrm{c}}, \\
		\bar{\mu}_{2,\mathrm{p}} &=  2r \mu_{2, \mathrm{c}},
	\end{align}
\end{subequations}
which is a linear system of four scalar equations for four scalar variables with a unique solution.

It remains to show that the dual variables $\mu$ always depend continuously on the system configuration $q$. This follows directly from the result in \cite{semple} which states that if for a convex parametric program, Slater's constraint qualification applies, the set of primal solutions is bounded for a certain parameter and dual variables are unique in a neighborhood of this parameter, then the Lagrange multipliers behave continuously. Since we already inferred that Lagrange multipliers are unique for any $q$, Slater's constraint qualification applies and primal optimal solutions are naturally bounded, even for the case of patch contacts, we conclude that the Lagrange multipliers depend continuously on $q$.

\subsection{Proof of Proposition \ref{propunique}} \label{app:proof2}
As explained in the main part of the paper, the goal is to derive smooth expressions for the algebraic variables of the CLS (\ref{cls}) in terms of the differential variables and the control input. Consequently, an ODE with a smooth right-hand side is obtained that admits to unique solutions due to the Picard-Lindelöf theorem. To this end, we again consider the active sets for the constraints (\ref{distprob:b}) and (\ref{distprob:c}) given by
\begin{align*}
	&\mathcal{A}_{1,\mathrm{p}} = \{i \ \vert \ \mu_{1,\mathrm{p},i} > 0\}, \\
	&\mathcal{A}_{2,\mathrm{p}} = \{i \ \vert \ \mu_{2,\mathrm{p},i} > 0\},
\end{align*}
We distinguish between two cases.

Case 1: Let $q(t)$ solve (\ref{contactDAE}) on $(t_{1},t_{2})$ such that $\vert \mathcal{A}_{1,\mathrm{p}} \vert = 2$ or $\vert \mathcal{A}_{2,\mathrm{p}} \vert = 2$. Then there exist unique ($z, \mu$) solving (\ref{distprob}), i.e., the SDF in classically differentiable. In this case, one can discuss the contact DAE similarly as done in \cite{tflag}, where a CLS with state jumps is discussed for the case of differentiable SDF.

Case 2: Let $q(t)$ solve (\ref{contactDAE}) on $(t_{1},t_{2})$ such that $\vert \mathcal{A}_{1,\mathrm{p}} \vert = 1$ and $\vert \mathcal{A}_{2,\mathrm{p}} \vert = 1$. Let 
\begin{alignat*}{2}
&\mathcal{A}_{1,\mathrm{p}} = \{k_{1}\}, &&\mathcal{A}_{2,\mathrm{p}} = \{k_{2}\}, \\
&k_{1} \in \{1,\dots,\nha\}, \quad &&k_{2} \in \{1,\dots,\nhb\}, \\
&a_{1} = A^{\top}_{1,k_{1}}, &&a_{2} = A^{\top}_{2,k_{2}}.
\end{alignat*}
Then the KKT conditions (\ref{statKKTcond}) imply
\begin{equation*}
	R(\theta_{1})a_{1} \mu_{1,\mathrm{p},k_{1}} + R(\theta_{2})a_{2} \mu_{2,\mathrm{p},k_{2}} = 0.
\end{equation*}
Since the the rows of $A_{1}, A_{2}$, are assumed to be unit vectors, the above equation reduces to 
\begin{equation} \label{k1}
	R(\theta_{1})a_{1} + R(\theta_{2})a_{2}  = 0.
\end{equation}
Differentiating (\ref{k1}) with respect to time yields
\begin{equation}\label{k2}
	\begin{split}
		0 &= \ddt	(R(\theta_{1})a_{1} + R(\theta_{2})a_{2}) \\
		&= R\left(\theta_{1} + \frac{\pi}{2}\right)a_{1} \omega_{1} + R\left(\theta_{2}+ \frac{\pi}{2}\right)a_{2}\omega_{2}.
	\end{split}
\end{equation}
By substituting (\ref{k1}) in (\ref{k2}), we obtain
\begin{equation*} 
	\begin{split}
		&0 = R\left(\frac{\pi}{2}\right) (R(\theta_{1})a_{1} \omega_{1} - R(\theta_{1})a_{1} \omega_{2}) \\
		\Leftrightarrow \quad&0= \omega_{1}-\omega_{2}
	\end{split}
\end{equation*}
and, consequently, 
\begin{equation*} 
	\begin{split}
		0 &= \dot{\omega}_{1}-\dot{\omega}_{2}.
	\end{split}
\end{equation*}
From the right-hand side of (\ref{cls:b}), we can then infer that it has to hold
\begin{alignat}{2}
	&  \phantom{{}\Leftrightarrow \, p \times \ntr \, \, {}}0 &&=M^{-1}_{\omega_{1}}(f_{\mathrm{v},\omega_{1}}(q,v,u) + (p-c_{1}) \times \ntr \lamn)  \nonumber \\
	& && \phantom{{}={}}- M^{-1}_{\omega_{2}}(f_{\mathrm{v},\omega_{2}}(q,v,u) - (p-c_{2}) \times \ntr\lamn) \nonumber \\
	&\Leftrightarrow \, \phantom{{}p \times \ntr \, \, {}} 0 &&= (M^{-1}_{\omega_{1}} + M^{-1}_{\omega_{2}})p \times \ntr \lamn \nonumber \\
	& && \phantom{{}={}}- C_{\mathrm{n}} \lamn +  R_{1}(q,v,u) \nonumber \\
	&\Leftrightarrow \, p \times \ntr \lamn &&= - \frac{R_{1}(q,v,u)}{M^{-1}_{\omega_{1}} + M^{-1}_{\omega_{2}}} +\frac{C_{\mathrm{n}} \lamn}{M^{-1}_{\omega_{1}} + M^{-1}_{\omega_{2}}}, \label{rhseqcond}
\end{alignat}
with
\begin{equation*}
C_{\mathrm{n}} = M^{-1}_{\omega_{1}} c_{1} \times \ntr + M^{-1}_{\omega_{2}} c_{2} \times \ntr.
\end{equation*}
The system velocities are captured in the vector \mbox{$\nu = (v_{1,\mathrm{x}},v_{1,\mathrm{y}},\omega_{1},v_{2,\mathrm{x}},v_{2,\mathrm{y}},\omega_{2})\in \mathbb{R}^{6}$} and \mbox{$M_{\omega_{1}}$, $M_{\omega_{2}}$, $f_{\mathrm{v},\omega_{1}}$, $f_{\mathrm{v},\omega_{2}}$}, denote components of the inertia matrix and the nominal dynamics influencing the respective elements of the system velocity. We used that $M$ is assumed to be diagonal. Furthermore, $R_{1}(q,v,u)$ denotes a smooth remainder term given by
\begin{equation*} 
	R_{1}(q,v,u) = M^{-1}_{\omega_{1}}f_{\mathrm{v},\omega_{1}}(q,v,u) -M^{-1}_{\omega_{2}}f_{\mathrm{v},\omega_{2}}(q,v,u).
\end{equation*}
Under the assumption of lasting patch contact, we have \mbox{$\dot{\mu} =0$}, cf. Case 3 in the proof of Proposition \ref{uniquepdvar}. It then follows 
\begin{equation}
	\dntr = R\left( \frac{\pi}{2}\right) R(\theta_{1})a_{1}\omega_{1}.
\end{equation}
We deduce
\begin{align}
	\dot{n}^{\top}v &= \begin{pmatrix}
		\dntr \\ 
		(\dot{p}-v_{1}) \times \ntr + (p-c_{1}) \times \dntr \\
		-\dntr \\ 
		-(\dot{p}-v_{2}) \times \ntr - (p-c_{2}) \times \dntr
	\end{pmatrix}^{\top}v \\
	&= R_{2}(q,v) + (\dot{p}\times \ntr + p\times \dntr)(\omega_{1}-\omega_{2}) \\
	&= R_{2}(q,v),
\end{align}
with 
\begin{equation*}
	R_{2}(q,v) = \begin{pmatrix}
		\dntr \\ 
		-v_{1} \times \ntr + c_{1} \times \dntr \\
		-\dntr \\ 
		v_{2} \times \ntr - c_{2} \times \dntr
	\end{pmatrix}^{\top}v.
\end{equation*}
Now let
\begin{equation*}
	n = n_{\mathrm{c}} + n_{\mathrm{p}} = \begin{pmatrix}
		\ntr \\ 
		-c_{1} \times \ntr \\
		-\ntr \\ 
		c_{2} \times \ntr
	\end{pmatrix} + \begin{pmatrix}
		0 \\ 
		p \times \ntr \\
		0 \\ 
		-p \times \ntr
	\end{pmatrix}.
\end{equation*}
Differentiating the contact normal with respect to time yields
\begin{align}
	0 &=\ddt	n^{\top}v \nonumber \\
	&= \dot{n}^{\top}v + n^{\top}\dot{v} \nonumber \\
	&=R_{2}(q,v) + n^{\top}(M^{-1}(f_{\mathrm{v}} + n\lamn)) \nonumber \\
	& = R_{2}(q,v) + (n_{\mathrm{c}} + n_{\mathrm{p}})^{\top}(M^{-1}(f_{\mathrm{v}} + (n_{\mathrm{c}} + n_{\mathrm{p}})\lamn))\nonumber  \\
	& = R_{2}(q,v) + n_{\mathrm{c}}^{\top}(M^{-1}(f_{\mathrm{v}} + (n_{\mathrm{c}} + n_{\mathrm{p}})\lamn))\nonumber \\
	& =R_{3}(q,v,u) + n_{\mathrm{c}}^{\top}M^{-1}n_{\mathrm{c}}\lamn + n_{\mathrm{c}}^{\top}M^{-1}n_{\mathrm{p}}\lamn\nonumber \\
	& =R_{3}(q,v,u) + n_{\mathrm{c}}^{\top}M^{-1}n_{\mathrm{c}}\lamn 
-C_{\mathrm{n}} p \times \ntr \lamn, \label{dnvcond}
\end{align}
with
\begin{equation*}
	R_{3}(q,v,u) = R_{2}(q,v) + n_{\mathrm{c}}^{\top}M^{-1}f_{\mathrm{v}}.
\end{equation*}
By substituting (\ref{rhseqcond}) into (\ref{dnvcond}), one obtains
\begin{align}
	0 &= R_{3}(q,v,u) + n_{\mathrm{c}}^{\top}M^{-1}n_{\mathrm{c}}\lamn \nonumber \nonumber\\
	&\quad -C_{\mathrm{n}} \left(\frac{C_{\mathrm{n}} \lamn - R_{1}(q,v,u)}{M^{-1}_{\omega_{1}} + M^{-1}_{\omega_{2}}} \right) \nonumber \\
	\Leftrightarrow \, 0 &= R_{4}(q,v,u) + \left(n_{\mathrm{c}}^{\top}M^{-1}n_{\mathrm{c}} - \frac{C_{\mathrm{n}}^{2}}{M^{-1}_{\omega_{1}} + M^{-1}_{\omega_{2}}} \right) \lamn,
\end{align}
with
\begin{equation*}
	R_{4}(q,v,u) = R_{3}(q,v,u) + \frac{C_{\mathrm{n}}}{M^{-1}_{\omega_{1}} + M^{-1}_{\omega_{2}}} R_{1}(q,v,u).
\end{equation*}
If it holds
\begin{equation} \label{ineqproof}
	\left(n_{\mathrm{c}}^{\top}M^{-1}n_{\mathrm{c}} - \frac{C_{\mathrm{n}}^{2}}{M^{-1}_{\omega_{1}} + M^{-1}_{\omega_{2}}} \right) > 0,
\end{equation}
an analytic expression for $\lamn$ is given by
\begin{equation} \label{lamnexplicit}
	\lamn  = -\left(n_{\mathrm{c}}^{\top}M^{-1}n_{\mathrm{c}} - \frac{C_{\mathrm{n}}^{2}}{M^{-1}_{\omega_{1}} + M^{-1}_{\omega_{2}}} \right)^{-1}  R_{4}(q,v,u),
\end{equation}
which is arbitrarily smooth in the states $q,v$ and the input $u$. \\
We now additionally use, that $p$ has to be an element of the hyperplane induced by the active halfspace constraints. I.e., it holds
\begin{equation*}
	\ntr^{\top}p = b_{1,k_{1}}.
\end{equation*}
Together with (\ref{rhseqcond}) we arrive at the explicit representation
\begin{equation} \label{explprep}
	p = \begin{pmatrix}
		\ntr^{\top} \\
		[\ntr ]_{\times}^{\top}
	\end{pmatrix}^{-1} \begin{pmatrix}
		b_{1,k_{1}} \\
		\frac{C_{\mathrm{n}}}{M^{-1}_{\omega_{1}} + M^{-1}_{\omega_{2}}}  - \frac{ R_{1}(q,v,u)}{(M^{-1}_{\omega_{1}} + M^{-1}_{\omega_{2}})\lamn} 
	\end{pmatrix},
\end{equation}
where we define $[\ntr ]_{\times} = (n_{\mathrm{tr},2},-n_{\mathrm{tr},1})^{\top}$ as then it follows $p \times \ntr = [\ntr ]_{\times}^{\top}p$. The right-hand side of (\ref{explprep}) is well-defined since we have assumed $\lamn > 0$ and since $\lamn$ admits a smooth expression in the states $q,v$ and the input $u$, so does the right-hand side of the above system for $p$. 

Therefore, we can substitute the derived explicit representations for primal and dual optimal variables of the distance problem as well as the force magnitude into the differential equations (\ref{cls:a})-(\ref{cls:b}). Uniqueness then follows as in Case 1 as we derived a pure ODE by index reduction with arbitrarily smooth right-hand side for which such a solution exists by the Picard-Lindelöf theorem.

It remains to show (\ref{ineqproof}). To this end, let
\begin{alignat*}{2}
	&s_{1} =  c_{1} \times \ntr, \quad &&s_{2} =  c_{2} \times \ntr, \\
	&m_{1} =  M^{-1}_{\omega_{1}}, &&m_{2} =  M^{-1}_{\omega_{2}}.
\end{alignat*}
By expanding all terms, we deduce
\begin{align*}
	n_{\mathrm{c}}^{\top}&M^{-1}n_{\mathrm{c}} - \frac{C_{\mathrm{n}}^{2}}{M^{-1}_{\omega_{1}} + M^{-1}_{\omega_{2}}} \\
	&= \ntr^{\top}M^{-1}_{v_{1}}\ntr + \ntr^{\top}M^{-1}_{v_{2}}\ntr + m_{1} s_{1}^{2} + m_{2} s_{2}^{2} \\
	&\phantom{{}={}}- \frac{m_{1}^{2}  s_{1}^{2} +m_{2}^{2}  s_{2}^{2} + 2m_{1}m_{2} s_{1}s_{2}  }{m_{1}+ m_{2}} \\
	& > m_{1} s_{1}^{2} + m_{2} s_{2}^{2} - \frac{m_{1}^{2}  s_{1}^{2} +m_{2}^{2}  s_{2}^{2} + 2m_{1}m_{2} s_{1}s_{2}  }{m_{1}+ m_{2}} \\
	&= \left(1 - \frac{m_{1} + m_{2} - m_{2}}{m_{1} + m_{2}}\right)m_{1}s_{1}^{2} \\
	&\phantom{{}={}}+ \left(1 - \frac{m_{2} + m_{1} - m_{1}}{m_{1} + m_{2}}\right)m_{2}s_{2}^{2} - \frac{2m_{1}m_{2} s_{1}s_{2}  }{m_{1}+ m_{2}} \\
	&=  \frac{m_{1}m_{2}}{m_{1} + m_{2}}s_{1}^{2} + \frac{m_{1}m_{2}}{m_{1} + m_{2}}s_{2}^{2} - \frac{m_{1}m_{2}}{m_{1} + m_{2}} 2s_{1}s_{2} \\
	&= \frac{m_{1}m_{2}}{m_{1} + m_{2}}( s_{1} - s_{2} )^{2} \\
	&\geq 0.
\end{align*}

\bibliographystyle{ieeetr}
\bibliography{Literature}

\end{document}